\documentclass[12pt]{article}

\usepackage{etex}

\usepackage{pgfplots}
\usepgfplotslibrary{fillbetween}
\pgfplotsset{compat=1.13}

\usepackage{latexsym}
\usepackage{amsfonts, amsmath}
\usepackage{amssymb}
\usepackage{theorem}
\usepackage[english]{babel}
\usepackage{longtable}
\usepackage{enumitem}
\usepackage{todonotes}

\usepackage{tikz}
\usetikzlibrary{arrows,decorations.pathreplacing, decorations.pathmorphing, positioning, calc}
\usepackage{caption}
\usepackage{subcaption}

\usepackage{algorithm}
\usepackage{algorithmic}

\usepackage{graphicx} 

\usepackage{float}

\usepackage[margin=2cm]{geometry}

\newtheorem{theorem}{Theorem}[section]
\newtheorem{definition}[theorem]{Definition}
\newtheorem{remark}[theorem]{Remark}
\newtheorem{proposition}[theorem]{Proposition}
\newtheorem{corollary}[theorem]{Corollary}
\newtheorem{lemma}[theorem]{Lemma}
\newtheorem{conjecture}[theorem]{Conjecture}

\newtheorem{observation}[theorem]{Observation}
\newtheorem{claim}[theorem]{Claim}
\newtheorem{stmnt*}{Statement}

\def\ND{\succsim\kern-11pt/\kern5pt}

\newenvironment{proof}
{\begin{trivlist}\item[]{Proof:}}{\hfill{$\square$}\noindent\end{trivlist}}

\usepackage{tikz}

\newcommand{\shouter}{{s}}
\newcommand{\shoutleft}{{s}_{\rm left}}
\newcommand{\chooser}{{c}}
\newcommand{\Lumpy}{{\sf Lumpy}}

\begin{document}
\title{Fair division of graphs and of tangled cakes}

\author{\normalsize{
Ayumi Igarashi\footnote{National Institute of Informatics;
e-mail: ayumi\_igarashi@nii.ac.jp
}\ 
\hspace{0.4mm} and William S. Zwicker\footnote{Union College, Department of Mathematics;
e-mail: zwickerw@union.edu
} $^,$\footnote{We thank Brenda Johnson, Dominik Peters, and Walter Stromquist.} 
}}

\maketitle

\begin{abstract}
A tangle is a connected topological space constructed by gluing together  several copies of the unit interval $[0, 1]$. We explore which tangles \emph{guarantee} envy-free (aka EF) allocations of connected shares for $n$ agents, meaning that such allocations exist no matter which monotonic and continuous functions represent agents' valuations.  Each single tangle $\mathcal{T}$ corresponds in a natural way to an infinite topological class $\mathcal{G}(\mathcal{T})$ of multigraphs, infinitely many of which are graphs.  This correspondence links EF fair division of tangles to EFk$_{\emph{outer}}$ fair division of graphs: the vertices of a graph, treated as indivisible objects, are allocated to the agents, each agent's share of vertices is contiguous (connected as an induced subgraph), no agent envies another's share after she pretends some selection of $k$ or fewer vertices disappear from that share, and that disappearance does not destroy contiguity. We know from Bil\`o et al \cite{Bilo} that all Hamiltonian graphs guarantee EF1$_{\emph{outer}}$ allocations when the number of agents is $2$, $3$, or $4$ and guarantee EF2$_{\emph{outer}}$ allocations for arbitrarily many agents.  
        
We show that exactly six tangles are stringable; these guarantee EF allocations of connected shares for any number of agents, and their associated topological classes contain only Hamiltonian graphs. 
 Any non-stringable tangle 
 $\mathcal{T}$
 has a finite upper bound $r$ 
 on the number of agents for which EF 
 allocations of connected shares are guaranteed.  Most graphs in the 
 associated non-stringable topological class $\mathcal{G}(\mathcal{T})$ are not Hamiltonian, and a \emph{negative transfer theorem} shows that for each $k \geq 1$ most of these graphs fail to guarantee EFk$_{\emph{outer}}$ allocations of vertices for $r + 1$ or more agents. 
This answers a question posed in Bil\`o et al \cite{Bilo}, and explains why a focus on Hamiltonian graphs was necessary for certain results in that paper.   
        
With bounds on the number of agents, however, we obtain positive results  for some non-stringable classes. An elaboration of Stromquist's moving knife procedure shows that the non-stringable lips tangle $\mathcal{L}$  guarantees envy-free allocations of connected shares for three agents. We then modify the discrete version of Stromquist's procedure in Bil\`o et al \cite{Bilo} to show that all graphs in the topological class $\mathcal{G}(\mathcal{L})$ (most of which are non-Hamiltonian) guarantee EF1$_{\emph{outer}}$ allocations for three agents.
\end{abstract}


\section{Introduction: Three Contexts for Fair Division}\label{INTRO}
Fair Division has been studied in a variety of contexts. In the \emph{continuous} setting, a single freely divisible good or ``cake," often modeled by the $[0,1]$ interval, is to be partitioned into $n$ pieces, with each of $n$ agents $i$ allocated a different piece $T_i$ of the partition. A traditional assumption is that each agent uses a countably additive and non-atomic measure $\mu_i$ to determine their own value for a piece of cake, but for our results here the more general assumption of \emph{monotone continuous valuations} $\nu_i$ (see precise definition in Appendix C) suffices. One concern in this setting is with the existence of connected and \emph{envy-free} allocations for the agents, in which each piece $T_i$ is connected,\footnote{For $[0,1]$, but not for other spaces we discuss later, a piece is connected if and only if it is a single subinterval.} and each agent values their assigned piece at least as highly as the pieces awarded to others: $\nu_i (T_i) \geq  \nu_i (T_j)$ for all $i,j$.
  
In the following classical cake-cutting result (see Stromquist \cite{Stromquist} and Simmons and Su \cite{Su}) the word ``guarantees" conveys that such envy-free connected allocations exist for every possible assignment $\{ \nu _i \}_{1 \leq i \leq n}$ of monotone continuous valuations to the agents:

\begin{theorem}\label{UrTheorem} 
The closed interval $[0,1]$ guarantees existence of connected and envy-free allocations for every finite $n$ number of agents with monotone continuous valuations.  
\end{theorem} 

\noindent The non-constructive proof uses Sperner's lemma, followed by a limit argument.  

In the \emph{indivisible objects} setting, we begin with a finite set $O$ of objects, with each individual object to be awarded in its entirety to some one agent. A simple assumption in this context is that each agent $i$ has an \emph{additively separable} valuation over subsets of $O$; $\nu _i$ assigns a non-negative real-number value to each $o \in O$, and the value of $\nu _i$ on a subset $S \subseteq O$ is determined by adding the values that $\nu _i$ assigns to the individual objects belonging to $S$. Again, we need less; it suffices to assume that agents' valuations are \emph{monotone}: $\nu_i(\emptyset ) = 0$ and $\nu_i(X) \leq \nu_i(Y)$ whenever $X \subseteq Y$.  

An allocation partitions $O$ into $n$ subsets $A_i$, with each of $n$ agents allocated a different piece $A_i$ of the partition.    
This setting is fundamentally ``lumpy" in a way that rules out envy-freeness as a reasonable goal for an allocation -- think of several agents sharing a set $O$ containing a single indivisible object $o$.  If each agent assigns strictly positive value to $o$ then as any allocation allocates $o$ to a single agent $j$, all other agents will envy $j$'s share. However, a notion due to Budish \cite{Budish} nicely walks around the lumpiness obstacle. An allocation $A = \{ A_i \}_{1 \leq i \leq n}$
is \emph{envy-free up to k goods}, aka EF$k$, if for each pair $i \neq j$ of agents there exists a selection of $k$ or fewer objects in agent $j$'s share $A_j$ such that any envy that agent $i$ may feel for $A_j$ would be erased by pretending those objects had been removed from $A_j$.  A variety of methods have subsequently been used to prove that an EF$1$ allocation is guaranteed to exist -- that is, such an allocation exists for any set $O$ and any sequence $\{ \nu _i \}_{1 \leq i \leq n}$ of $n$ monotone valuations for the $n$ agents. For example, in Lipton, Markakis, Mossel, and Saberi \cite{Lipton} we find the following result\footnote{In the case of $n$ additively separable valuation functions for the $n$ agents, one can obtain an EF$1$ allocation through maximizing the \emph{Nash product}; if $A = \{A_i\}_{1 \leq i \leq n}$ is an allocation of the items in the finite set $O$ that maximizes the $\Pi_{1 \leq i \leq n}\nu_i(A_i)$ of agents' valuations of their shares, then $A$ is EF$1$. 
}:

\begin{theorem}
For the case of $n$ agents with monotone valuation functions over subsets of a finite set $O$, an EF$1$ allocation of the items in $O$ exists. 
\end{theorem}

\emph{Graph fair division}, our third context, was introduced very recently, in Bouveret, Cechl\'arov\'a, Elkind, Igarashi and Peters \cite{Bout}. Here, the vertices of a connected graph $G=(V,E)$ are viewed as indivisible objects, and we insist that each agent be allocated a \emph{contiguous} share -- meaning that the subgraph induced by the subset $A_i \subseteq V$ of vertices allocated to them be connected.\footnote{We will use \emph{connected} for the continuous setting, and \emph{contiguous} for the discrete context of graphs.} We again assume monotonicity for the valuations $\nu_i$ that agents use to determine the worth of subsets $S$ of vertices. 
More than one fairness notion has subsequently been applied to graphs, but our focus here is on a variant of Budish's EF$k$ notion: 

\begin{definition}
An allocation $A = \{A_i\}_{1 \leq i \leq n}$ of the vertex set of a graph $G = (V,E)$ is \emph{envy-free up to $k$ outer goods}, aka EF$k_\emph{outer}$, if for each pair $i \neq j$ of agents there exists a selection of $k$ or fewer objects in agent $j$'s share $A_j$ such that any envy that agent $i$ may feel for $A_j$ would be erased by pretending those objects had been removed from $A_j$, and such that the selected objects would leave $A_j$ contiguous (as an induced subgraph), were they to be removed. 
\end{definition}

\noindent This was the notion used in Bil\`o, Caragiannis, Flammini, Igarashi, Monaco, Peters, Vinci, and Zwicker \cite{Bilo}, where the following result can be found, though not in the exact form stated here:\footnote{The version in Bil\`o et al \cite{Bilo} separates out stronger results for the special cases of exactly $2$ agents, exactly $3$, and exactly $4$. We discuss some of those later. }
%

\begin{theorem}\label{BiloTheorem}
The class of path graphs guarantees existence of contiguous and \emph{EF1}$_\emph{outer}$ allocations for two, three, or four agents with monotone valuations, and guarantees existence of contiguous and \emph{EF2}$_\emph{outer}$
allocations for five or more agents with monotone valuations.    
\end{theorem} 

\noindent Of course, adding more edges to a path graph cannot hurt the contiguity of any agent's share, so Theorem \ref{BiloTheorem} applies as well to the class of all \emph{Hamiltonian} (aka \emph{traceable}) graphs. How about non-Hamiltonian graphs? 

\medskip 

\noindent $[\clubsuit]$ \emph{``For the case of three or more agents, it is a challenging open problem to find an infinite class of non-traceable graphs that guarantee EF1."} [from the introduction to \cite{Bilo}, with ``EF1" here actually referring to EF$1_\emph{outer}$]

\medskip 

\noindent Our investigations  can be seen as an immediate follow-up to \cite{Bilo} -- one aimed directly at addressing this open problem.  Here,  we provide two somewhat different answers (which are restated more precisely in Section \ref{Conclusions}):

\medskip 

\noindent $[\clubsuit$, Answer \#1] There exist infinite \emph{topological classes} of non-traceable graphs that guarantee EF$1_\emph{outer}$ for three agents (and also for two).

\medskip  

\noindent $[\clubsuit$, Answer \#2] 
There exists no infinite \emph{topological class} of non-traceable graphs that guarantees EF$k_\emph{outer}$ for arbitrarily many agents (and this holds for each integer $k \geq 1$).

\medskip 

Why focus on classes rather than on individual graphs, and what is a \emph{topological class}? Dominik Peters (private communication) had identified several small, non-Hamiltonian graphs (not discussed in \cite{Bilo}) that seemed to also guarantee EF$1_\emph{outer}$ for three agents, but these appeared to be sporadic, fitting no discernible pattern.\footnote{Any graph at all guarantees EF$1_\emph{outer}$ for $n$ agents, when $n$ is at least as great as the number of vertices, but sporadic examples may exist for other reasons as well.} Our methods here rest on introducing a fourth setting for fair division, involving continuous spaces more complicated than the $[0,1]$ interval. This fourth context -- which we call \emph{tangles} -- seems to be a new one for fair division, although a very similar model was proposed at about the same time in Bei and Suksompong \cite{Warut}.\footnote{There is little  overlap between their results and ours, primarily because \cite{Warut} assumes that singular points of a tangle can be allocated to more than one agent, and because the focus is not on envy-freeness or its close relatives.} 
Our original motivation for studying tangles stemmed from what they might tell us about fair division of graphs, which proves to be quite a bit, at least when  the notion of fairness under consideration is some version of envy-freeness. In much the same way that the $[0,1]$ interval corresponds to the class of all \emph{path graphs} (obtained by adding vertices at $0$, at $1$, and at finitely many points in the interval's interior) each individual tangle $\mathcal{T}$ corresponds to an infinite class $\mathcal{G}(\mathcal{T})$ of multigraphs, most of which are graphs. The \emph{negative transfer theorem} then converts any negative result for envy-free fair division of $\mathcal{T}$ into a corresponding result for EF$k_{\emph{outer}}$ fair division of the graphs in $\mathcal{G}(\mathcal{T})$.  

\begin{figure}[!ht]
{\centering
\includegraphics[height=64mm,width=98mm]{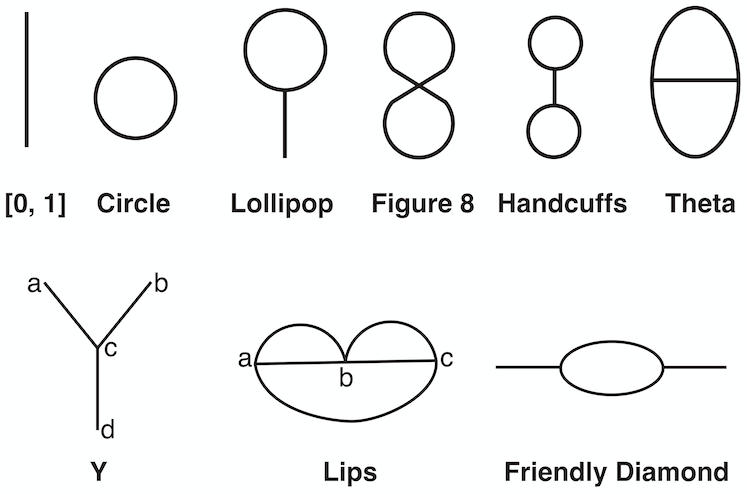}
\caption{Nine tangles.  Top row contains the six basic stringable tangles.}}
\label{fig:basic:stringable}
\end{figure}

But fair division of tangles has turned out to be pretty interesting on its own, and so this paper can equally well be classified as an introduction to a new context for fair division. Most of our results flow from tangles to graphs, and this determines the organization of the rest of the paper. In Section \ref{TNGL}, we introduce fair division of tangles, and use \emph{cutsets} to prove the Tangle Boundedness Theorem \ref{OnlyStringable}:  while the six \emph{basic stringable} tangles each satisfy a version of Theorem \ref{UrTheorem} (because each is essentially similar to the unit interval), every other  tangle comes with an upper bound on the number of agents for which connected and envy-free allocations are guaranteed. Part of that proof is deferred to Section \ref{DegSeqSection}, where a case argument shows that every realizable degree sequence for a tangle leads either to one of the six stringable tangles, or to a key inequality used in the Theorem \ref{OnlyStringable} proof.  

Our more interesting positive results for tangles are covered in Section \ref{StromLips}, where we show that the lips tangle $\mathcal{L}$ (see Figure $1$
) guarantees connected and envy-free allocations for three agents, via an argument that elaborates Stromquist's well-known moving knife method \cite{Stromquist} over $[0,1]$; the additional complexity for $\mathcal{L}$ arises from the need to take explicit account of the ways in which $\mathcal{L}$'s arcs are connected. Similar arguments apply to a number of other tangles, and a classification theorem for three agents may be within reach. Such a classification does exist for two agents: Envy-free and connected allocations are guaranteed if and only if the tangle has a \emph{bipolar ordering}. The result and  proof are quite parallel to the characterization via \emph{bipolar numbering} for graphs and two agents in \cite{Bilo}, and thus represents a transfer from the discrete to the continuous realm.
It follows, for example, that the friendly diamond and lips tangles $\mathcal{L}$ both guarantee connected, envy-free allocations for two agents. Details are in Section \ref{TangleForTwo}.

In section \ref{GeneralizedGap}, we provide improved upper bounds on the number of agents for which non-stringable tangles guarantee connected and envy-free allocations, based on generalizing the notion of \emph{type II trident} from \cite{Bilo}. These bounds suggest a conjecture that would settle, for each tangle $\mathcal{T}$ and positive integer $n$, whether or not $\mathcal{T}$ guarantees connected and envy-free allocations for $n$ agents. 

In Section \ref{TanglesToGraphs}, we take a tangle 
$\mathcal{T}$, pretend it is a graph, and then add an arbitrary selection of degree-$2$ subdivision vertices along edges, to generate the infinite \emph{topological class} $\mathcal{G}(\mathcal{T})$ of graphs (and multigraphs) associated to $\mathcal{T}$. A natural directed partial ordering generates a \emph{final segment filter} over $\mathcal{G}(\mathcal{T})$, accompanied by a notion of ``almost all."  Almost all members of $\mathcal{G}(\mathcal{T})$ are then simple graphs. Moreover, if $\mathcal{T}$ is not one of the six basic stringable tangles, then \vspace{-3.5mm}
\begin{itemize}
\item almost all these graphs are non-Hamiltonian, and 
\vspace{-2mm}
\item there exists an upper bound $n'$ such that for any positive integer $k$ and any $n > n'$ almost all graphs in $\mathcal{G}(\mathcal{T})$ fail to guarantee contiguous EF$k_\emph{outer}$ allocations for $n$ agents.  
\end{itemize}
\vspace{-3.5mm}
We refer to this last result as the \emph{Negative Transfer Principle}.

Stromquist's original moving knife argument applies in the ``classical" cake-cutting setting (when $[0,1]$ is being divided) where it guarantees connected and envy-free allocations for three agents; to obtain the same result for four or more agents seems to require the less constructive approach via Sperner's Lemma. A discretized version of Stromquist's argument, in \cite{Bilo}, applies to path graphs, where it guarantees contiguous and EF$1_{\emph{outer}}$ allocations of path graphs for three agents. In section \ref{StromDiscrete}, we combine elements of that argument with the version of Stromquist for the lips tangle $\mathcal{L}$ (from Section \ref{StromLips}) to show that every graph in the topological class $\mathcal{G}(\mathcal{L})$ of lips (in which almost all graphs are non-Hamiltonian) similarly guarantees EF$1_{\emph{outer}}$ allocations for three agents.

Finally, in the concluding section \ref{Conclusions}, we tie the results from earlier sections more precisely to our two answers (stated earlier) to question $\clubsuit$, and point to some remaining open questions.



\section{A Fourth Context for Fair Division: Tangles}\label{TNGL}
%
%
\begin{definition} \label{TangleDef} A \emph{tangle} $\mathcal{T}$ is a connected topological space that can 
be built by taking finitely many disjoint copies of the closed interval $[0,1]$ and identifying\footnote{
\label{CW} See, for example, the wikipedia article on
quotient spaces, \emph{aka} indentification spaces. Tangles are essentially the same as the topologies of finite, one-dimensional CW complexes. Equivalently, to get a tangle take any ``drawing" of a connected 
multigraph in $\Re ^3$ that uses portions of circular arcs as edges, has finitely many vertices, and avoids any edge 
intersections (except at vertices shared by the edges concerned), treat it as a subset of  $\Re ^3 $ 
and endow it with the topology induced by that of $\Re ^3 $.} some of their endpoints.
\end{definition} 

\noindent When two tangles are homeomorphic as topological spaces, we consider them to be identical -- in particular, this means that different instructions for building tangles from multiple copies of $[0,1]$ can yield identical tangles.\footnote{Unlike a CW complex, a tangle does not remember how it was built.} For example, the lollipop tangle of Figure $1$ 
can be obtained from two copies of $[0,1]$ by identifying any three of the four endpoints; alternately the stick of the lollipop could be built from more than one copy, as could the loop. Other tangles (see Figure $1$
) include the unit interval, the circle, the $Y$, the figure 8, the handcuffs, the $\Theta$, the lips, and the friendly diamond. Informally, one might say that tangles are graphs made continuous.

\begin{definition} \label{SingularPointDef}
Let $\mathcal{T} $ be a tangle.  For any point $x \in \mathcal{T}$ a small 
neighborhood of $x$ looks like $k$ copies of the half open interval $[0,1)$ with all the $0$s 
identified -- that is, $k$ open rays spreading out from $x$. We will refer to $k$ as the \emph{degree} of $x$ in $\mathcal{T}$, notated $\emph{deg}(x)$; points of degree $1$ or degree at least $3$  are \emph{singular}, and those of degree $2$ are 
\emph{nonsingular}. An \emph{edge} of $\mathcal{T}$ is a connected component in the space obtained when all singular points of $\mathcal{T}$ are removed.   


\end{definition}

\noindent Note 
that any point $x$ that belongs to one of the copies of $[0,1]$ (from Definition \ref{TangleDef}) and is not an endpoint of that copy 
must wind up as a non-singular point, so each tangle contains only a finite number of singular points. It's easy to see that each edge of $\mathcal{T} $ is homeomorphic to the open interval $(0,1)$ (with the exception of the circle tangle, which has no singular points and one edge), and that the number of  edges corresponds to the minimum number of copies of $[0,1]$ needed to build $\mathcal{T} $, in the sense of Definition \ref{TangleDef}.

The simplest tangle of all -- the unit interval -- has served as the mathematical model of a \emph{cake} in the study of \emph{cake cutting}, the informal term used within one major thread in the literature on fair division.\footnote{This explains our use of ``tangled cakes" in the title.} Our set-up for fair division of tangles will match that discussed for $[0,1]$ in the previous section. We assume that each of finitely many agents $i$ has a monotone continuous valuation  over connected pieces of $\mathcal{T}$.
Our main concern here is with the existence of connected and \emph{envy-free} allocations for $n$ agents, in which $\mathcal{T}$ is partitioned into $n$ pieces, each agent $i$ is awarded a different piece $T_i$ of the partition, each piece is topologically connected (as a subset of the topological space $\mathcal{T}$),
and each agent values their assigned piece at least as highly as the pieces awarded to others: $\mu_i (T_i) \geq  \mu_i (T_j)$ for all $i,j$. 

What happens, then, when some other tangle stands in for $[0,1]$ in Theorem \ref{UrTheorem}; for which tangles, and which values of $n$, are connected and envy-free allocations for $n$ agents guaranteed?  We begin by identifying a special sub-class of tangles. Intuitively, a \emph{stringable} tangle is made by arranging a single length of string on a table so that it does not touch itself, except that either or both of the two ends may butt up against each other, or against some other part of the string (or do both). A moment's thought reveals that this idea is not a perfect analogue of Hamiltonian path, or of eulerian path, in the continuous context.  

\begin{definition}\label{StringableDef}
A tangle $\mathcal{T}$ is \emph{stringable} if there exists a continuous surjection $F\!\!: [0,1] \rightarrow \mathcal{T}$ from the closed unit interval to $\mathcal{T}$ that is injective when restricted to the open interval $(0,1)$.

\end{definition}

\noindent 

\noindent The reader will not find it difficult to come up with six examples of stringable tangles:

\begin{definition}\label{BasicStringableDef}
The \emph{basic stringable tangles} are the six shown in the top row of Figure 1: the closed unit interval itself, the circle, the lollipop, the figure 8, the handcuffs, and the $\Theta$.  
\end{definition}  

\noindent We will leave it to the reader to confirm that these six are indeed stringable. In Section \ref{DegSeqSection}, we demonstrate that 
no other stringable tangles exist.  Until then, the term ``basic"  is used to designate these six in particular (and has no other special meaning). 
 
%

Each stringable tangle can substitute for $[0,1]$ in Theorem \ref{UrTheorem}, but only because stringable tangles are almost the same as $[0,1]$.  The extension is unremarkable, and turns out to add nothing to our understanding of graph division. 

\begin{proposition} \label{StringableProp} Each stringable tangle guarantees existence of connected and envy-free allocations for any finite number of agents with monotone continuous valuations.
\end{proposition}

\begin{proof} Let $F\! \! : [0,1] \rightarrow \mathcal{T}$ witness $\mathcal{T}
$'s stringability, and $\{\nu_i \}_{1 \leq i \leq n}$ be agents' valuations over connected subsets of $
\mathcal{T}$.  These induce corresponding valuations $\{\nu_i^\star \}_{1 
\leq i \leq n}$ on $[0,1]$ defined by $\nu_i^\star (X) = \nu_i (F[X])$. Apply 
Theorem \ref{UrTheorem} to obtain an allocation $\{T_i^\star \}_{1 \leq i \leq 
n}$ of connected subsets of $[0,1]$ to the agents, that is envy-free with 
respect to the $\nu_i^\star$ . The corresponding allocation                $\{ F[T_i^\star ] \}_{1 \leq i \leq n}$ of connected subsets of $\mathcal{T}$ will not quite form a partition of $\mathcal{T}$, in the following circumstances:

\vspace{3 mm}

\noindent \underline{Case 1}: Assume that for some $j \neq k$, $F$ maps an endpoint of $[0,1]$ in $T_j 
^\star$ and an interior point of $[0,1]$ in $T_k ^\star$ to the same element 
$x$ of $\mathcal{T}$, so that $x \in F[T_j^\star ] \cap F[T_k^\star]$.  
In this case, remove the endpoint from $T_j 
^\star$ and remove  $x$ from  $F[T_j ^\star ]$, while 
allowing $x$ to remain a member of $F[T_k ^\star ]$.

\vspace{3 mm}

\noindent \underline{Case 2}: Assume that Case 1 does not hold, but $F(0) = x = F(1)$ with $0 \in 
T_j^\star$, $1 \in T_k^\star$, and $j \neq k$. In this case, remove $0$ from $T_j^\star$ and remove $x$ from $F[T_j ^\star ]$, while 
allowing $x$ to remain in $F[T_k ^\star ]$.  

\vspace{3 mm}

\noindent These adjustments convert $\{ F[T_i^\star ] \}_{1 \leq i \leq n}$ into a true allocation $\{T_i \}_{1 \leq i \leq n}$ of $\mathcal{T}$ to the agents. Moreover, removing an endpoint of $[0,1]$ from a connected region $T_i ^\star$ of $[0,1]$ can never disconnect $T_i ^\star$. Thus, the $T_i$ are images, under a continuous function, of connected subsets of $[0,1]$, and  remain connected as subsets of $\mathcal{T}$.  As the $\nu_i$ are non-jumpy (see Remark \ref{jumpyremark} in Appendix C), we have that $\nu_i ^\star(T_i^\star) =  \nu_i (F[T_i^\star ])= \nu_i (T_i)$ so that our allocation $\{T_i \}_{1 \leq i \leq 
n}$ inherits its envy-freeness from that of $\{T_i^\star \}_{1 \leq i \leq 
n}$. \end{proof}


 More striking, and consequential for graphs, is that stringable tangles are the \emph{only}  ones that satisfy Theorem \ref{UrTheorem}:


\begin{theorem}\label{OnlyStringable} $\mathbf{[Tangle}$ $\mathbf{Boundedness}$ $\mathbf{Theorem]}$ For any tangle $\mathcal{T}$ that is not one of the six basic stringable tangles, there exists an upper bound $n_0$ on the number of agents with monotone continuous valuations for which envy-free allocations are guaranteed.
\end{theorem}



\noindent As a consequence of Proposition \ref{StringableProp} and Theorem \ref{OnlyStringable}, we have:

\begin{corollary}\label{OnlySix} The only stringable tangles are the six basic ones.
\end{corollary}

\noindent  Theorem \ref{OnlyStringable} follows immediately from Lemmas \ref{G2L}, \ref{NSGap}, and \ref{DegSeqLemma}.  (Note, however, that the proof of Lemma \ref{DegSeqLemma} is deferred to Section \ref{DegSeqSection}, where the relevant material on degree sequences is discussed).
Much rests on the notion of \emph{gap} $\geq 2$ \emph{cutset}: 

\begin{definition} Let $\mathcal{T} $ be a tangle.  
If $X \subseteq \mathcal{T}$ has $t$ elements and $\mathcal{T}\setminus X$ is disconnected with $t+k$ connected components then we will say that the \emph{gap} $\gamma (X)$ of $X$ is equal to $k$ and that $X$ is a \emph{gap} $k$ \emph{cutset of cardinality} $t$; if $\gamma(X) \geq 2$, then $X$ is a \emph{gap} $\geq 2$ \emph{cutset of cardinality} $t$.
\end{definition}

\begin{lemma}  \underline{Gap $\geq 2$ Lemma} \label{G2L} \;
Let $\mathcal{T} $ be a tangle. 
Suppose $\mathcal{T}$ has a gap $\geq 2$ cutset of cardinality $t$. Then, for each  $n \geq t+1$, $\mathcal{T}$ does not guarantee connected envy-free allocations for $n$ agents with identical monotone continuous valuations.   
\end{lemma} 
\begin{proof} 
Let $\mathcal{T} $ be a tangle, $X \subseteq \mathcal{T}$ have $t$ elements, $\mathcal{T}\setminus X$ have $t + k$ connected components $D_1, \dots$, $D_{t+k-1},D_{t+k}$ with $k \geq 2$, and $n \geq t+1$ be the number of agents.  We will construct a common valuation for which no connected EF allocation exists.  For each $i=1,2,\dots$, $t+k-1$ spread $\frac{t+1}{t+k-1} \leq 1$ units of value over $D_i$.  Spread $n-t \geq 1$ units of value over $D_{t+k}$.  Each agent now places a total value of $n+1$ on the whole of $\mathcal{T}$, and will be envious of some other agent if his share is worth less than $\frac{n+1}{n} > 1$.   

Now let $\{T_i\}_{1 \leq i \leq n}$ be any allocation that assigns connected pieces to each agent.  Then there are at least $n - t$ agents whose share includes no members of the cutset $X$, so that for each of these agents their share must be a subset of some single connected components $D_1, \dots$, $D_{t+k-1},D_{t+k}$.  If for each of these agents their share is a subset of $D_{t+k}$, then as the value of $D_{t+k}$ is only $n-t$, some agent receives a share valued at $1$ or less, and is envious.  If some agent receives a share that is a subset of one of the other components, then their value for their share is also $1$ or less, and they are envious. \end{proof}


\noindent Notice that for the first three examples below, the gap $\geq 2$ cutset  mentioned consists of the singular points of degree $3$ or greater:\begin{itemize} 
\vspace{-3mm}
\item The $\mathcal{Y}$ tangle has a gap $2$ cutset of cardinality $1$; connected EF allocations to $n$ agents can fail for each $n \geq 2$.\footnote{The notion of gap $2$ cutset actually arose from this example, when (in a private communication) Walter Stromquist pointed to the ``$\mathcal{Y}$" tangle as an example that does not guarantee EF allocations for two agents.} 

\vspace{-2mm}

\item The friendly diamond tangle has a gap $2$ cutset of cardinality $2$, so connected EF allocations to $n$ agents can fail for each $n \geq 3$. Cardinality $1$ cutsets  each have gap $1$ or less, and we know that this tangle does guarantee connected EF allocations for $2$ agents (because it is the continuous analogue of a  graph having a bipolar numbering -- see \cite{Bilo} and Section \ref{TangleForTwo}). 

\vspace{-2mm}

\item The lips tangle has a gap $2$ cutset of cardinality $3$, so connected EF allocations to $n$ agents can fail to exist for each $n \geq 4$.  Cardinality $2$ cutsets each have gap $1$ or less, and we show that lips guarantees connected EF allocations for two agents (Section \ref{TangleForTwo}) and for three agents (Section \ref{StromLips}).  

\vspace{-2mm}

\item For the unit interval (and for every stringable tangle) cardinality $t$ cutsets each have gap $1$ or less, for every choice of $t$.  These tangles guarantee connected EF allocations for every number of agents (via the original Sperner's Lemma argument).

\vspace{-2mm}

\item A tangle that is itself disconnected (so, not actually a tangle, according to the definition we have given here) has a gap $2$ cutset of  cardinality $0$. It does not guarantee connected allocations to even a single agent. \end{itemize}

\begin{definition}  
For a tangle $\mathcal{T}$, let $\Sigma _3 (\mathcal{T})$ denote the set of singular points of $\mathcal{T}$ having degree $3$ or greater, $\sigma _3 (\mathcal{T})$ denote $|\Sigma _3 (\mathcal{T})|$, and $\epsilon (\mathcal{T})$ denote the  number of edges of $\mathcal{T}$. 
\end{definition}

\begin{lemma}\label{NSGap}  
For any tangle $\mathcal{T}$ satisfying $\epsilon (\mathcal{T}) - \sigma_3 (\mathcal{T}) \geq 2$, $\Sigma_3 (\mathcal{T})$ is a cutset of gap $\geq 2$.
\end{lemma}
\begin{proof} We claim that deleting all  points in 
$\Sigma_3 (\mathcal{T})$ disconnects any tangle $\mathcal{T}$ (stringable or not) into its constituent edges.  It is clear that any edge connecting two singular points in $\Sigma_3 (\mathcal{T})$ is disconnected from everything else, and this is also evident if the edge connects one singular point in $\Sigma_3(\mathcal{T})$ with another having degree $1$; of course, if the edge connects two degree $1$ singular points, then $\mathcal{T}$ is necessarily the unit interval, and the claim holds trivially.  \end{proof}

\begin{lemma}\label{DegSeqLemma}
For any tangle $\mathcal{T}$ other than the six basic stringable ones, $\epsilon (\mathcal{T}) - \sigma_3 (\mathcal{T}) \geq 2$.
\end{lemma}

\noindent The proof of Lemma \ref{DegSeqLemma}, via degree sequences, is postponed to Section \ref{DegSeqSection}.
Note that the Tangle Boundedness Theorem (\ref{OnlyStringable}) now  follows immediately from Lemmas \ref{G2L}, \ref{NSGap}, and  \ref{DegSeqLemma}.
%

Most tangles, we have seen, are not stringable, and the Tangle Boundedness Theorem \ref{OnlyStringable} paints only part of the picture.  What more can we say? 

%
\begin{definition}
The \emph{gap threshold} of a tangle is the smallest integer $t$ for which a  gap $\geq 2$ cutset of cardinality $t$ exists; if no gap $\geq 2$ cutset exists of any size, then the gap threshold is $\infty$.  
\end{definition}

The gap threshold for the $\mathcal{Y}$ tangle is $1$, for the friendly diamond is $2$, and for the lips is $3$; for these cases the set of singular points of degree $3$ or $4$  is the relevant cutset, as the reader can verify. The \emph{Gap} $2$ \emph{Lemma} now tells us that the friendly diamond does not guarantee connected envy-free allocations for three agents or more, while the lips tangle does not guarantee connected envy-free allocations for four or more. Stringables, of course, have gap threshold $\infty$. Moreover each non-stringable tangle has a finite gap threshold, which serves as an upper bound $n_0$ for Theorem \ref{OnlyStringable}.  But does the gap threshold provide the \emph{least} upper bound?  Sometimes it does: 


\begin{remark}\label{GapThresholdRemark} For each tangle $\mathcal{T}$ appearing in Figure $1$
,  $\mathcal{T}$ guarantees envy-free connected allocations for $n$ agents whenever $n$ is less than or equal to $\mathcal{T}$'s gap threshold.
\end{remark}

\noindent The remark summarizes several \emph{positive} results. In Section \ref{TangleForTwo}, we characterize those tangles that guarantee EF allocations of connected shares for two agents, in terms of the existence of a \emph{bipolar ordering}.  
The characterization and its proof are quite parallel to the closely related characterization for graphs and two agents in \cite{Bilo}, and thus represents a transfer from the discrete to the continuous realm.\footnote{Note that our other transfer results move in the opposite direction.} It follows that the friendly diamond tangle and the lips tangle $\mathcal{L}$ each guarantee connected, envy-free allocations for two agents.
     Moreover, in section \ref{StromLips}, we show that $\mathcal{L}$ guarantees connected and envy-free allocations for three agents.
This result, taken along with those described in the previous paragraph, the mentioned values for some specific gap thresholds, and the assertions about stringable tangles before the statement of Theorem \ref{OnlyStringable}, collectively verify Remark \ref{GapThresholdRemark}. 

Remark \ref{GapThresholdRemark} does not refer to all tangles, so it does not provide a general converse to the Gap $2$ Lemma.  In section \ref{GeneralizedGap}, we provide examples of tangles for which the minimal $n_0$ value for Theorem \ref{OnlyStringable} is strictly less than the gap threshold, along with generalizations of the \emph{gap 2 cutset} and \emph{gap threshold} concepts suggested by the examples, which lead to:

\begin{conjecture} \label{GapThresholdConjecture} 
Each tangle $\mathcal{T}$ guarantees envy-free connected allocations for $n$ agents with monotone continuous valuations whenever $n$ is less than or equal to $\mathcal{T}$'s \emph{generalized gap threshold} (as defined in Section \ref{GeneralizedGap}).
\end{conjecture}

\noindent In Corollary \ref{Tangle2Car} we show that Conjecture \ref{GapThresholdConjecture} is correct for the special case in which the generalized gap threshold is equal to two.  The full conjecture  would be a powerful positive result; in conjunction with Theorem \ref{OnlyStringable}, it would completely settle, for any specified number $n$ of agents, exactly which tangles guarantee envy-free allocations for $n$ agents. 
There are reasons to suspect that even if the conjecture is true, proving such a ``wholesale" positive result would require new techniques. 

\section{Degree sequences for tangles} \label{DegSeqSection}

Recall that one key step in the proof of the Tangle Boundedness Theorem was Lemma \ref{DegSeqLemma}, which asserts every tangle 
 $\mathcal{T}$ is either one of the six basic stringble tangles, or satisfies the inequality:
\begin{equation} 
\epsilon (\mathcal{T}) - \sigma_3 (\mathcal{T}) \geq 2,
\end{equation} 

\noindent where $\epsilon (\mathcal{T})$ is the number of $\mathcal{T}$'s edges and $\sigma_3 (\mathcal{T})$ is the number of singular points of degree three or greater.  Our proof of the lemma here proceeds in the form of two claims, which use \emph{degree sequences} as topological invariants for tangles.  These are almost the same as degree sequences for graphs and multigraphs (which are well-known invariants in the graph context):

\begin{definition} 
If $\mathcal{T}$ is a tangle and $i \geq 1$ is any integer, let $d_i$ be the number of singular points of degree $i$
(So that $d_2$ is, perforce, $0$).\footnote{We ruled out singular points of degree two by fiat, in Definition \ref{SingularPointDef}.  Note that \emph{all} tangles have the same number of degree 2 points -- continuum many -- so counting them does not help distinguish among tangles. 
}
The \emph{degree sequence} $\overrightarrow{deg}(\mathcal{T})$ \emph{of a tangle} is the vector $\langle d_1, d_2, \dots, d_r  \rangle$, where $d_r > 0$ and $d_s = 0$ for all $s > r$.  
\end{definition}

Note that throughout this section, we cannot presume Corollary \ref{OnlySix} (the only stringable tangles are the six basic ones). To do so would be circular. 

\begin{claim}\label{SixBasicSeq}  For any tangle $\mathcal{T}$:
\begin{enumerate}

\item If $\overrightarrow{deg}(\mathcal{T}) =\langle \, \rangle$, then $\mathcal{T}$ must be the circle,

\item If $\overrightarrow{deg}(\mathcal{T}) =\langle 2 \rangle$, then $\mathcal{T}$ must be the unit closed interval,

\item If $\overrightarrow{deg}(\mathcal{T}) =\langle 1,0,1 \rangle$, then $\mathcal{T}$ must be the lollipop

\item If $\overrightarrow{deg}(\mathcal{T}) =\langle 0,0,2 \rangle$, then $\mathcal{T}$ must either be the handcuffs or the theta, and

\item If $\overrightarrow{deg}(\mathcal{T}) =\langle 0,0,0,1 \rangle$, then $\mathcal{T}$ must be the figure 8.

\end{enumerate}
\end{claim}

\begin{proof} We will prove part $d$ and leave the remainder to the reader.  If there exists a continuous function $F\! \! :\! [0,1] \rightarrow \mathcal{T}$
 that is injective on $(0,1)$, and satisfies $F(0)=F(1)$ then we will refer to the image of $[0,1]$ 
 in $\mathcal{T}$ as a \emph{cycle} of tangle $\mathcal{T}$.  We claim that if $\mathcal{T}$ has no degree $1$ singular points then it must contain a cycle. Start a path at any point in $\mathcal{T}$ and continue it as far as possible without allowing the path to visit any point a second time.  Such a path must get ``stuck" at some point $z$, for one of two reasons.  Perhaps there is one way in to $z$ and no ways out -- in that case $z$ is a singular point of degree $1$, but $\mathcal{T}$ has none of those.  Otherwise, it must be that $z$ has been visited earlier by the path, in which case the portion $\mathcal{C}$ of the path between the first and second visit to $z$ must be a cycle.
 
Now assume that $\overrightarrow{deg}(\mathcal{T}) =\langle 0,0,2 
\rangle$, and let $\mathcal{C} \subseteq \mathcal{T}$ be a cycle.  Clearly 
$\mathcal{T}$ must contain points not in $\mathcal{C}$. As $\mathcal{T}$ is 
connected, there is some path $\mathcal{P} \subseteq \mathcal{T}$ that starts at 
some point $x\in \mathcal{C}$ and ends at  some point $y\in \mathcal{T} 
\setminus \mathcal{C}$. Without loss of generality assume that $x$ is the only point in $\mathcal{P} \cap \mathcal{C}$ -- so that $x$ is a singular point of degree three or more -- and also assume that $\mathcal{P}$ does not yet visit any point twice.  Extend $\mathcal{P}$ in $\mathcal{T}$ until it either revisits some point $z \in \mathcal{P}$ for a second time, or hits a point $z \in \mathcal{C}$, noting that $z \neq x$ (else $x$ has degree four or more).  After this extension $\mathcal{P} \cup \mathcal{C}$ is a tangle, with two singular points $x$ and $z$, each of degree three. If $z\in \mathcal{C}$ then $\mathcal{P} \cup \mathcal{C}$ is $\Theta$ and if not then $\mathcal{P} \cup \mathcal{C}$ is handcuffs.  Finally, if there were any additional points of $\mathcal{T}$ outside $\mathcal{P} \cup \mathcal{C}$ they would be connected to $\mathcal{P} \cup \mathcal{C}$ via some path that adds additional singular points, or changes the degree of $x$ or of $z$, so no such additional points exist.
\end{proof}

\begin{claim}\label{NSGapR}  
\emph{[Lemma \ref{DegSeqLemma}]} For any tangle $\mathcal{T}$ other than the six basic stringable ones, $\epsilon (\mathcal{T}) - \sigma_3 (\mathcal{T}) \geq 2$.
\end{claim}
\begin{proof} For graphs, and multigraphs, we know the number of edges is half the sum of the degrees of the vertices, and it is easy to see that this holds for tangles, as well,\footnote{The circle tangle is the unique exception, as its single edge does not contribute to the degree of any vertex.} with \emph{singular points} replacing \emph{vertices}.  Thus, from the degree sequence of a tangle $\mathcal{T}$ we can derive both the number $\epsilon (\mathcal{T})$ of edges and the number $\sigma_3 (\mathcal{T})$ of singular points of degree three or greater.  The relationship also allows us to eliminate degree sequences that are \emph{parity blocked} because they imply an odd sum.

The rest of the proof breaks all parity unblocked degree sequences of tangles into nine cases, with each case either falling under Claim \ref{SixBasicSeq} (so leading only to a basic stringable tangle) or shown via a very short calculation (as sketched in the preceding paragraph) to imply $\epsilon (\mathcal{T}) - \sigma_3 (\mathcal{T}) \geq 2$.  Note that  parameters $r, s,$ and $t$ below are assumed to be nonnegative integers. We leave it to the reader to check that the nine cases are exhaustive.

\medskip

\noindent \underline{Case 1} $\overrightarrow{deg}(\mathcal{T}) = \langle  \rangle$: See Claim \ref{SixBasicSeq}.

\medskip

\noindent \underline{Case 2} $\overrightarrow{deg}(\mathcal{T}) = \langle r \rangle$: We must have $r\leq 2$, lest $\mathcal{T}$ be disconnected; for $r = 2$ see Claim \ref{SixBasicSeq}. 

\medskip

\noindent \underline{Case 3} $\overrightarrow{deg}(\mathcal{T}) = \langle 0,0,2 \rangle$ or $ \langle 1,0,1 \rangle$:  See Claim \ref{SixBasicSeq}.

\medskip

\noindent \underline{Case 4} $\overrightarrow{deg}(\mathcal{T}) = \langle t,0,1 \rangle$, $t \geq 3$: In this case   $\epsilon (\mathcal{T}) - \sigma_3 (\mathcal{T})  \geq \frac{t+3}{2} -1 \geq 3-1 = 2$

\medskip

\noindent \underline{Case 5} $\overrightarrow{deg}(\mathcal{T}) = \langle 0,0,2r \rangle$, $r \geq 2$: Here  $\epsilon (\mathcal{T}) - \sigma_3 (\mathcal{T}) \geq \frac{6r}{2} -2r =r \geq2$

\medskip

\noindent \underline{Case 6} $\overrightarrow{deg}(\mathcal{T}) = \langle t,0,2r+1 \rangle$, $t \geq 1$, $r \geq 1$: Then   $\epsilon (\mathcal{T}) - \sigma_3 (\mathcal{T})  \geq \frac{3(2r+1) +1} {2} -(2r+1) =   r+1   \geq2$.

\medskip

\noindent \underline{Case 7} $\overrightarrow{deg}(\mathcal{T}) = \langle 0,0,0,1 \rangle$:  See Claim \ref{SixBasicSeq}.

\medskip

\noindent\underline{Case 8} $\overrightarrow{deg}(\mathcal{T}) = \langle t,0,r,s \rangle$, $s \geq 1$, $t+r+s \geq 2$: \\ \vspace{2.5mm} \hspace{10mm}
Then   $\epsilon (\mathcal{T}) - \sigma_3 (\mathcal{T}) \geq$ 
 $ \big{ \lceil} \frac{t+3[r+(s-1)] +4} {2} -(r+s)\big{ \rceil} =   \big{ \lceil} \frac{t+r+s+1}{2}\big {\rceil } \geq 2$.


\noindent\underline{Case 9} $\mathcal{T}$ has $r+1 \geq 1$ singular points of degree $3$ or greater, at least one of which is degree $5$ or greater:
Then   $\epsilon (\mathcal{T}) - \sigma_3 (\mathcal{T})  \geq$
$ \big{ \lceil} \frac{5 +3r} {2} -(r+1)\big{ \rceil} =   \big{ \lceil} \frac{3+r}{2}\big {\rceil } \geq 2$.   \end{proof}

\section{Lips guarantees EF allocations for three agents}\label{StromLips}

Stromquist's moving knife argument \cite{Stromquist} handles three agents and produces an EF allocation of the $[0,1]$ interval, with connected shares. The proof of Theorem \ref{StromLipsTheorem} elaborates  that argument to deal with the more complex connectivity of the lips tangle.  We assume only that agents' valuations are \emph{monotone continuous}.  This assumption, detailed in Appendix C, is significantly weaker than the requirement that valuations be non-atomic and countably additive measures.

Lips  has three singular points, arranged in Fig 1 from left to right as $a$, $b$, and $c$ (with $a$ and $c$ of degree 3, and $b$ of degree 4).  There are two tangle edges from $a$ to $b$, two from $b$ to $c$, and a fifth edge from $a$ to $c$ (which loops below the other 4 edges in Fig 1 to form the ``lower lip").  Note that the lips tangle $\mathcal{L}$ is not stringable, so Theorem \ref{StromLipsTheorem} below establishes that stringable tangles are not the only ones that guarantee EF for three agents. 

\begin{theorem}\label{StromLipsTheorem}
For three agents with monotone continuous  
valuations over the lips tangle $\mathcal{L}$, an EF allocation with connected shares always exists. 
\end{theorem}

\begin{figure}[!ht]
{\centering
\includegraphics[height=57mm,width=78mm] {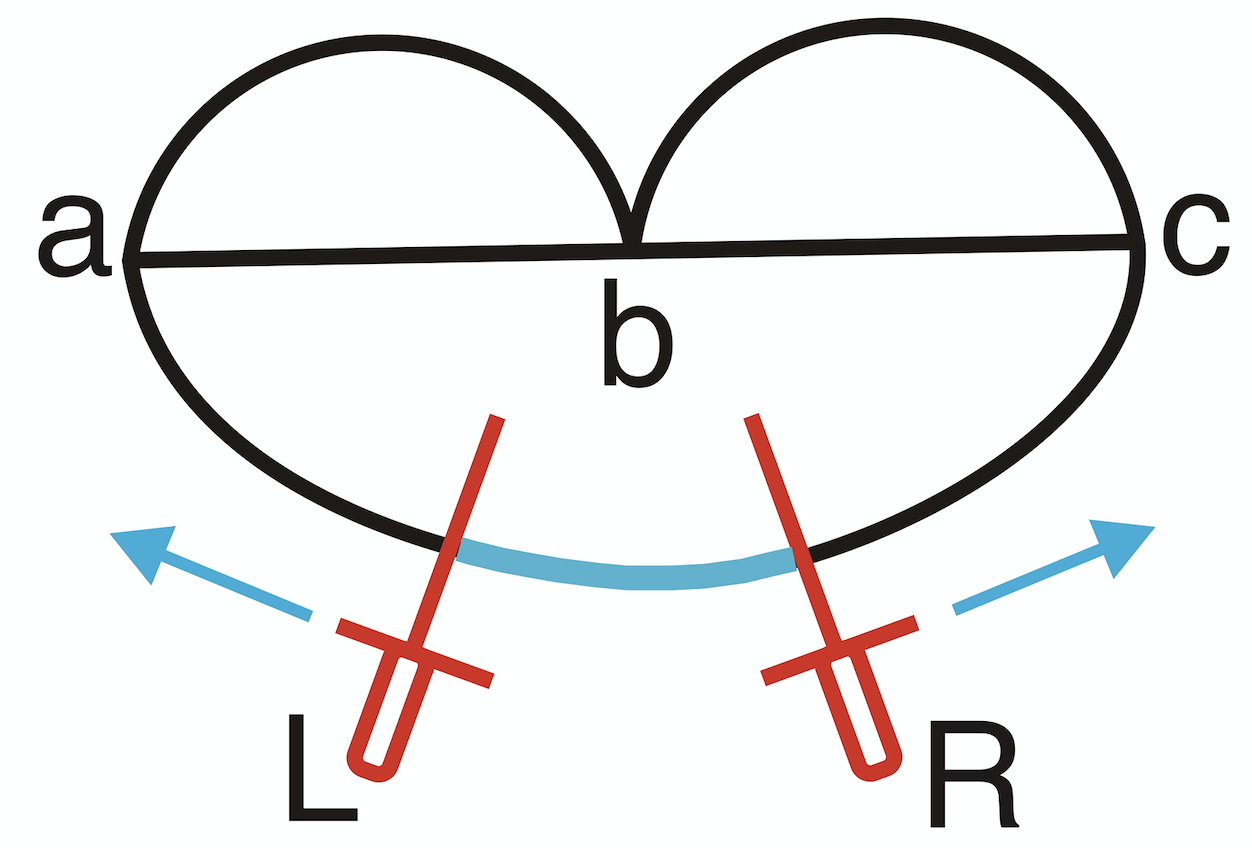}
\caption{Two swords move}}
\end{figure}

\begin{proof}  [INITIAL STAGE] Two swords, labeled $L$ and $R$ in Fig 2, both begin at the center of edge $a - c$ and maintain a radial orientation;  $L$ moves in a clockwise direction  (hence, initially to the left) while $R$ moves counterclockwise, in such a way that $L$ will reach $a$ at the same moment that $R$ reaches $c$.  However, for this initial stage of the argument we assume they have not yet reached $a$ and $c$.  The open-ended portion of $a - c$ lying \emph{strictly} between $L$ and $R$ is the \emph{central piece} of the partition; the part of $\mathcal{L}$ outside the central piece is the \emph{complement}.  

\begin{figure}[!ht]
\centering
\includegraphics[height=60mm,width=78mm] 
{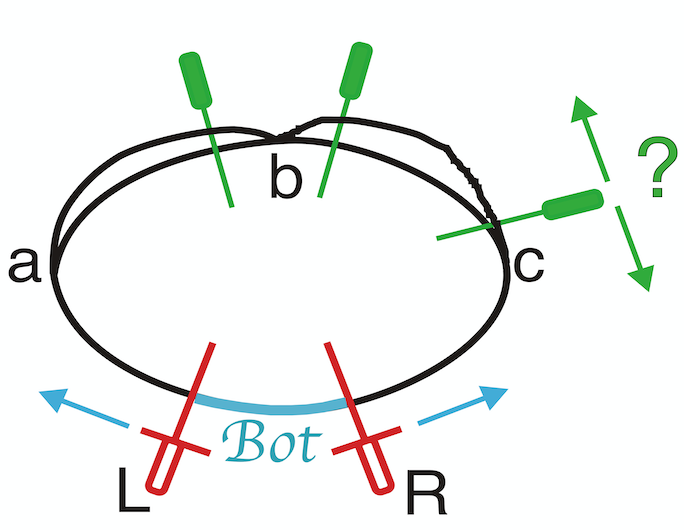}
\caption{Three knives also move \dots though who knows which directions they move in.}
\end{figure}

We imagine the two $a - b$ edges as almost pinched together, while the two $b - c$ edges are similarly pinched,  so that $\mathcal{L}$ looks roughly like a circle, as in Fig 3.  
As the two swords rotate, each of the three agents continually adjust the positions of their knives (which are always held radially to the ``circle") so as to divide the complement into two parts of equal value in their eyes.  If, during the initial stage, any one of these knives should cut the complement at any location, they would divide the complement into two connected pieces; loosely speaking, this is because the complement has a bipolar ordering (see Section \ref{TangleForTwo} -- this condition is analogous to a graph having a bipolar numbering).  Any agent can call ``cut" during the initial stage if she sees the central piece (labeled \emph{Bot}) as exactly equal in value to the larger of the two pieces formed when the complement is cut by the median knife.  If this happens then the procedure ends, with cuts made at $L$, $R$, and the median knife.  The caller is awarded the central portion, and the pieces of the complement are awarded to the other two agents as they would be in Stromquist's original procedure.

\noindent [SECOND STAGE] It may happen that the two swords reach $a$ and $c$ without anyone having called ``cut," as in Fig 4. In that case the swords pause at these locations \emph{with the points $a$ and $c$ still considered to be members of the complement}, while we identify the current location of the median knife.  Without loss of generality assume this knife is either at $b$ or counterclockwise of $b$ (as in Fig 4).
In this case, the left sword now remains fixed at $a$ while the right side continues its relentless counter-clockwise rotation, further enlarging the central piece as in Fig 5. Notice that in this case the median knife will rotate counter-clockwise (weakly, perhaps) in stage 2, so that during stage 2 it never passes from from strictly on one side of $b$ to strictly on the other side.  It may be that some agent calls ``cut" before the right sword reaches $b$.  In this case the procedure terminates, with cuts at $L$,  $R$, and the median knife.  Again, it is clear that the pieces are connected\footnote{
One must be careful, in case the median knife is exactly at $b$, to include $b$ itself in the piece to the counterclockwise side of $R$.
}, and that when the pieces are assigned to agents in the standard way, the allocation is EF.

\begin{figure}[!ht]
{\centering
\includegraphics[height=50mm,width=78mm] {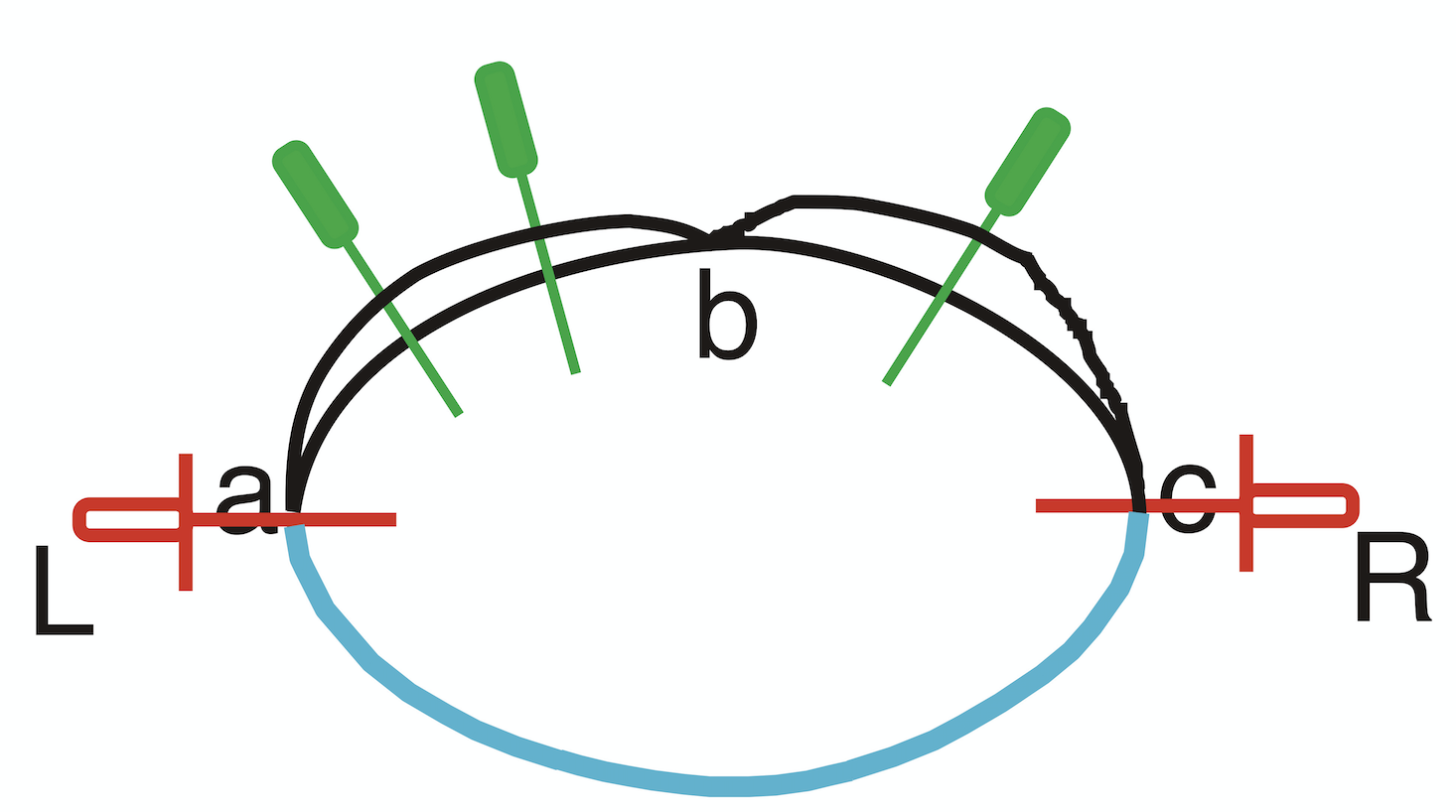}
\caption{The swords may reach $a$ and $c$}}
\end{figure}
 
 \begin{figure}[!ht]
{\centering
\includegraphics[height=57mm,width=78mm] {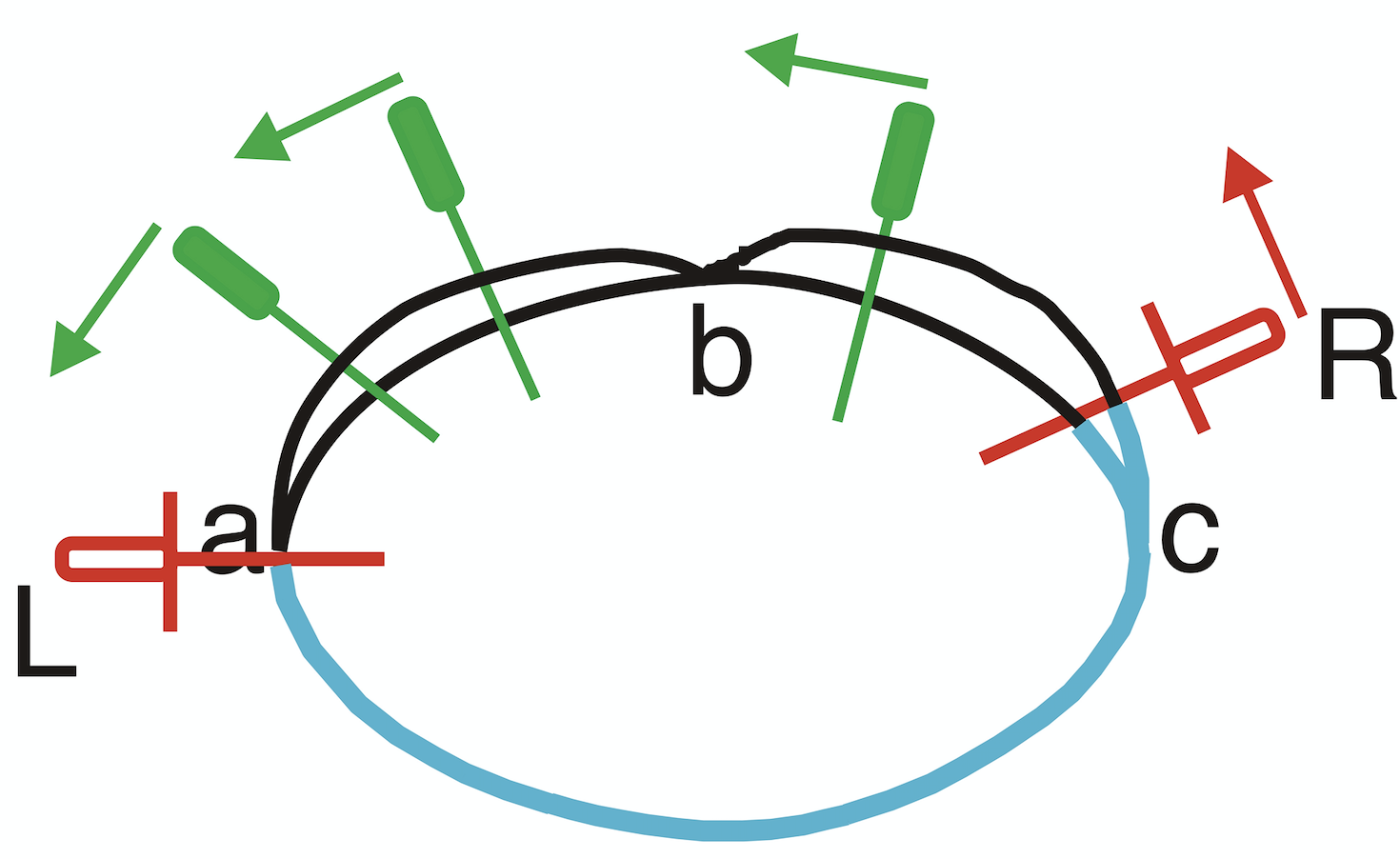}
\caption{The left sword halts while the right sword continues to rotate}}
\end{figure}

\noindent [THIRD STAGE] It may happen that in the second stage the right sword reaches $b$ and still no one has called ``cut."  At this point the complement is the closed loop consisting of the two $a - b$ edges (including the points $a$ and $b$ themselves).  It must be that all three agents have a strictly higher valuation for this loop than they have for the rest of $\mathcal{L}$, 
else someone would have called ``cut" in an earlier stage.  
With this \emph{loop value guarantee} in hand, we abandon the previous approach and launch a different version of Stromquist, as the third stage of the procedure. It is easier to picture this version by first introducing two disconnections -- or ``breaks" -- that convert $\mathcal{L}$ to $\mathcal{L}^\star$ (though one could instead work with the original topology on $\mathcal{L}$ by slightly altering the following description).\footnote{As topological spaces, $\mathcal{L}$ and $\mathcal{L}^\star$ share the same set of points, but the collection of open sets for $\mathcal{L}^\star$ is expanded from that for $\mathcal{L}$.}
 There is no harm in switching to $\mathcal{L}^\star$, because any shares connected in $\mathcal{L}^\star$ remain connected in $\mathcal{L}$.

First disconnect at $a$, yielding Fig 6.  Such a disconnection can be done in two ways, depending on which side of the break remains attached to $a$; Fig 6 indicates that $a$ remains attached to the $a - b$ loop by showing $a$ as a solid disk on the $a - b$ loop and as an open disk labeled $a^\star$ on the $a-c$ edge.  Second, disconnect at $b$, unrolling the $a-b$ loop, as in Fig 7 (and noting that $b$ itself remains attached to the $b-c$ loop).  Notice that the resulting friendly diamond tangle 
supports a \emph{bipolar ordering} (see Section \ref{TangleForTwo}).

 \begin{figure}[!ht]
{\centering
\includegraphics[height=57mm,width=78mm] {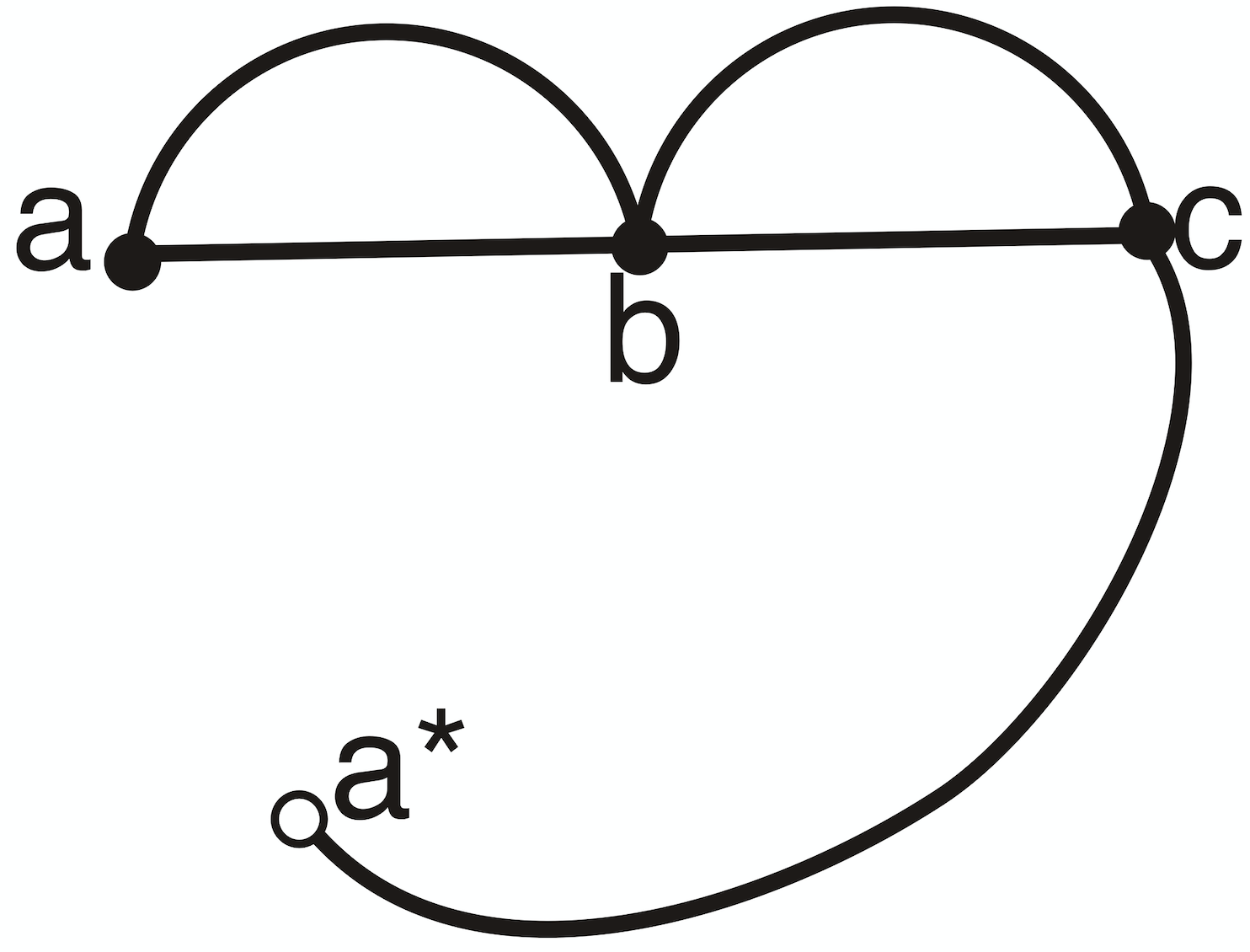}
\caption{The first disconnection}}
\end{figure}

 \begin{figure}[!ht]
{\centering
\includegraphics[height=80mm,width=78mm] {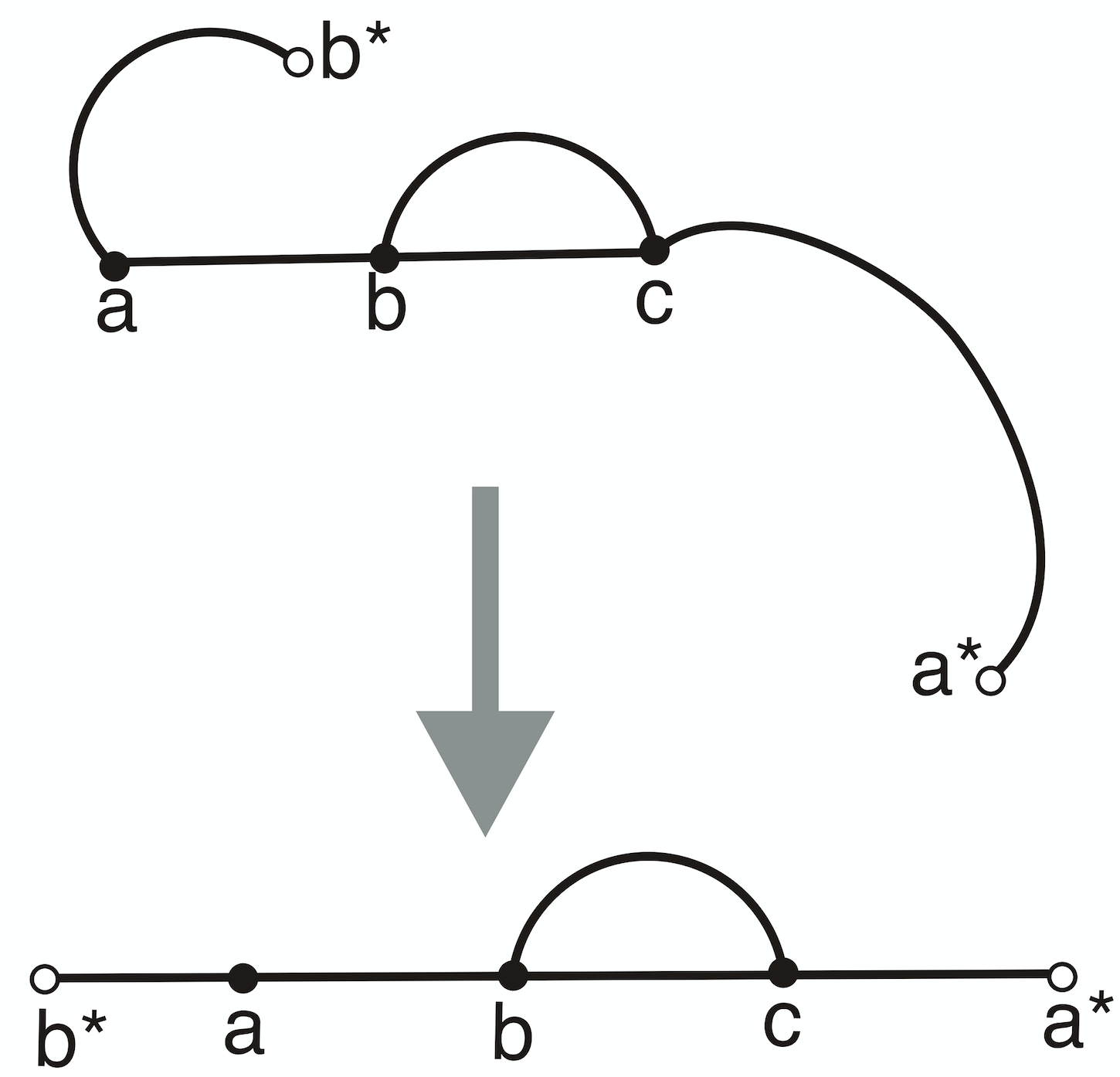}
\caption{The second disconnection}}
\end{figure}

We will treat $\mathcal{L}^\star$ as if it were an open line segment $(b^\star , a^\star)$ by squeezing the two $b-c$ edges almost together.  We proceed to use a standard Stromquist on this open segment, with the sword starting on the left (the $b^\star$ end).  We know that someone calls ``cut" before the sword reaches $b$, thanks to the loop value guarantee that applies whenever stage three is reached.  The  part of $\mathcal{L}^\star$ to the right of the sword cut still looks like the friendly diamond, and so it will consist of two connected pieces once it is cut by the median knife.  \end{proof}

\section{The generalized gap threshold of a tangle}\label{GeneralizedGap}
 We have seen that the \emph{gap threshold} of any non-stringable tangle $\mathcal{T}$ provides an upper bound on the number of agents for which envy-free and connected allocations of $\mathcal{T}$ are guaranteed. But it does not always yield the least such bound -- the \emph{agent bound}:  

\begin{definition}\label{AgentBoundDef}
The agent bound of a non-stringable tangle $\mathcal{T}$ is the greatest integer $n$ such that envy-free and connected allocations of $\mathcal{T}$ are guaranteed for $n$ agents with monotone continuous valuations.\footnote{
The definition does not say, ``guaranteed for $n$ \emph{or fewer} agents."  However, the Generalized Gap Threshold conjecture would imply equivalence between the definition as stated and the \emph{or fewer} variant.  
}
\end{definition}

\noindent Figure 8 shows the friendly diamond tangle along with a variant that has an agent bound strictly less than its gap threshold. The $\delta$\emph{-diamond tangle} $\mathcal{T}_{\delta \diamondsuit}$ replaces the single degree 3 singular point $a$ of the friendly diamond with triangle \emph{abc}, which is a \emph{subtangle} of $\mathcal{T}_{\delta \diamondsuit}$ -- that is, a subset consisting of one or more singular points, possibly together with some of the tangle edges joining these points.   
It is easy to check that, unlike the original friendly diamond, $\mathcal{T}_{\delta \diamondsuit}$ has no gap-2 cutset of cardinality 2, but $\{a,b,c\}$ provides one of cardinality 3. In a moment, we will argue that the agent bound is 2 for $\delta$-diamond; replacing $a$ with a triangle thus upped the gap threshold, but had no effect on the agent bound.

\begin{figure}[!ht]
{\centering
\includegraphics[height=50mm,width=78mm] {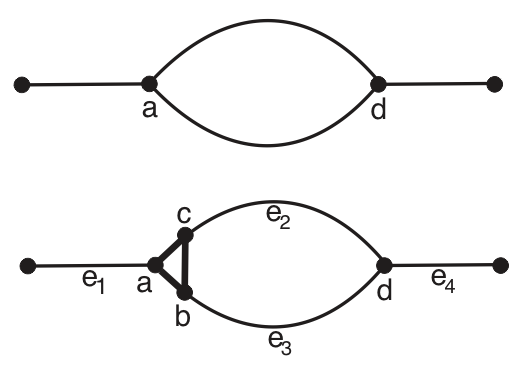}
\caption{The $\delta$-diamond tangle (below) with the friendly diamond (above).}}
\end{figure}

In section \ref{TangleForTwo}, we will review the two types of \emph{tridents} from \cite{Bilo}, and point out that a \emph{type I} trident is the same as a gap $\geq 2$ cutset of cardinality 1. More generally, a gap $\geq 2$ cutset with more than one point can be thought of as a collective form of type I trident. But the cutsets we have discussed up to now do not resemble collective versions of \emph{type II} tridents, or collective mixes of both types. So it should not be surprising to discover that we can obtain tighter estimates for the agent bound of a tangle by generalizing our cutset definition to allow collective tridents of mixed type.\footnote{
Perhaps, what should seem mildly surprising is that this generalization is not needed to establish that finite agent bounds exist for all non-stringable tangles.}  

Consider the closed subset $\triangle abc$ of $\mathcal{T}_{\delta \diamondsuit}$ consisting of the three points $a$, $b$, and $c$, together with the three small tangle edges joining these three points. Let $X_1 = \triangle abc$ and $X_2 = \{d\}$. We will argue that $X_{\delta \diamondsuit}=\{X_1 , X_2 \}$ serves as an example of the sort of generalization we have in mind; for three agents, it can block envy-free connected allocations of $\delta$-diamond in the same way that $X_{\diamondsuit}=\{a,d\}$ does for the original friendly diamond tangle.  


\begin{proposition}\label{DelDiamondProp}
The $\delta$-diamond tangle 
$\mathcal{T}_{\delta \diamondsuit}$ does not guarantee envy-free connected allocations for three or more agents with identical monotone continuous valuations.
\end{proposition}

\begin{proof}
Let $X_{\delta \diamondsuit}=\{X_1 , X_2 \}$ be defined as in the previous paragraph. An allocation of $\mathcal{T}_{\delta \diamondsuit}$ may award to some agent $i$ a piece $T_i$ that includes a majority (either two or three) of the three points $a$, $b$, and $c$; in this case we will say that $T_i$ \emph{dominates} $X_1$.  Clearly, at most one such piece can dominate $X_1$. Thus, with a connected allocation to three or more agents, there must be at least one agent $k$ whose piece $T_k$ neither dominates $X_1$ nor contains $d$ as a member. This $T_k$ must be a connected subset of the ``leftovers" $\mathcal{L}_k = \mathcal{T}_{\delta \diamondsuit}\setminus \cup _{i \neq k} T_i$ obtained
 by excising the other two pieces. But the four edges $e_1$, $e_2$, $e_3$, and $e_4$ that do not belong to $\triangle abc$ are disconnected in $\mathcal{L}_k$. If a path $\mathcal{P}$ in $\mathcal{T}_{\delta \diamondsuit}$ connects a point from one of these four edges to a point in a different one of the four, then it  passes either through $d$ or through at least two of the three points $a$, $b$, and $c$. Thus no such such path can be contained in $T_k$, and so $T_k$ contains elements from at most one of these four edges. 
%
%


The rest of the argument proceeds as in the proof of the Gap $\geq 2$ Lemma \ref{G2L}, treating the edges $e_1$, $e_2$, $e_3$, and $e_4$ as if they were the disconnected components $D_j$ in that earlier proof -- in the case of exactly three agents, this would mean assigning common valuations that spread one unit of value along each $e_j$, with some neglibible (or zero) value assigned to the edges of $\triangle abc$.\end{proof}

Our definition of \emph{generalized gap} $\geq$ \emph{2 cutset} is now as follows:

\begin{definition} \label{GenGapCutsetDef} A \emph{generalized gap} $\geq 2$ \emph{cutset of cardinality} $t$ is a finite set $X=\{X_1, X_2, \dots , X_t\}$ of pairwise disjoint connected subtangles\! \footnote{\label{SubTangleFN}
Note that a subtangle of $\mathcal{T}$ is necessarily closed in the topology of $\mathcal{T}$, so that deleting all points from a generalized cutset leaves connected components that are open in $\mathcal{T}$. The proof of Lemma \ref{GG2L} needs open components, to be sure that monotone continuous valuations can assign nonzero values to these components. See the final paragraph of this section for further discussion of the closed and connectedness requirements.
} of a tangle $\mathcal{T}$, for which the set $\mathcal{T} \setminus (\cup _{i=1}^m X_i )$ contains at least $t+2$ disconnected components $C_1, C_2 , \dots , C_{t}, C_{t+1}, C_{t+2}$ satisfying:
\begin{enumerate}
\item[$($a$)$] For each $i$ and $j$ the intersection $X_i \cap \overline{C_j}$ of $X_i$ with the closure (in $\mathcal {T}$) of $C_j$ is either empty, or contains a single point $s_{i,j}$, known as the \emph{contact point} for $X_i$ and $C_j$.
\item[$($b$)$] For each $i$, $X_i$ is either
\begin{itemize}
\item a ``type I" member, containing a single point of  $\mathcal{T}$, 
\item or is a ``type II" member: a set with more than one element, satisfying that exactly three components $C_j$ share a contact point with $X_i$, and that these contact points are distinct.
\end{itemize}
\end{enumerate}
\end{definition}

The proof of Proposition \ref{DelDiamondProp} now generalizes easily, to yield:

\begin{lemma}  \underline{Generalized Gap $\geq 2$ Lemma} \label{GG2L} \;
Let $\mathcal{T} $ be a tangle. 
Suppose $\mathcal{T}$ has a gap $\geq 2$ generalized cutset of cardinality $t$.  Then for each  $n \geq t+1$, $\mathcal{T}$ does not guarantee connected envy-free allocations for $n$ agents with identical monotone continuous valuations.
\end{lemma} 

Lemma \ref{GG2L} leads naturally to a definition and a conjecture:
\begin{definition}
The generalized gap threshold of a tangle $\mathcal{T}$ is the smallest integer $t$ for which there exists a gap $\geq 2$ generalized cutset of cardinality $t$ for $\mathcal{T}$.
\end{definition}

\begin{conjecture}\label{GenGapConj} \underline{The Generalized Gap Threshold Conjecture}
 \, A tangle $\mathcal{T}$ guarantees connected envy-free allocations for $n$ agents with monotone continuous valuations if and only if $n$ is no greater than $\mathcal{T}$'s generalized gap threshold.
\end{conjecture}

We show, in the next section, that Conjecture \ref{GenGapConj} holds for two agents -- it is the tangle version of (part of) the theorem in \cite{Bilo}, characterizing the class of graphs that guarantee contiguous EF$1_\emph{outer}$ allocations for two agents. We do not yet know the status of this conjecture for three agents, but it seems plausible that the moving-knife methods used in section \ref{StromLips} can be generalized to prove it, and so to characterize exactly which tangles guarantee envy-free allocations for three agents.  

Suppose there are four or more agents, however. Then it seems that any proof would require new techniques. In this case, moving knife methods seem unsuited for obtaining connected allocations (see Brams, Taylor, and Zwicker \cite{BTZ};  Barbanel and Brams \cite{BB}; and see the related conjecture in Br\^{a}nzei and Nisan \cite{BN})
 while Theorem \ref{UrTheorem} relies on Sperner's Lemma, which is tied to the topology of $[0,1]$ in particular. The only avenue we are aware of, then, for proving Conjecture \ref{GenGapConj} would seem to require versions of Sperner's Lemma for tangles other that $[0,1]$, and we are unaware that such versions exist.

Note that while Definition \ref{GenGapCutsetDef} requires type II elements of a generalized cutset to be subtangles, and be connected, the rationale for these requirements is that they render the definition more conceptually transparent. The Lemma \ref{GG2L} proof actually rests only on the weaker assumption (explained in footnote \ref{SubTangleFN}) that such elements be closed in the tangle. As we see in the next section, for the special case of $n=2$ agents imposing the stronger requirements in the definition does not affect the status of the \emph{Generalized Gap Threshold Conjecture}. Examples suggest that this may also hold for $n>2$. Of course, 
designing appropriate algorithms for calculating the \emph{Generalized Gap Threshold} would require us to know for certain whether demanding connected subtangles can affect the outcome of the calculation, and at this time that question remains open.

\section{Envy-free allocations of tangles for two agents: a classification theorem} \label{TangleForTwo}

Our goal is an existence theorem for tangles that is analogous to (and that builds on) the following result from Bil\`o et al \cite{Bilo} characterizing the graphs that guarantee connected, EF$1_{\emph outer}$  allocations for two agents (with some terminology explained below):

\begin{theorem}  \label{thmEQ} The following conditions are equivalent for every simple connected graph $G$:
	\begin{enumerate}
		\item $G$ admits a bipolar numbering.
		\item $G$ guarantees EF$1_{\emph{outer}}$ for two agents with monotone valuations.
		\item $G$ guarantees EF$1_{\emph{outer}}$ for two agents with identical, additive, binary valuations.
		\item $G$ does not contain a trident.
		\item The block tree $B(G)$ is a path.
	\end{enumerate}
\end{theorem}

\noindent We refer the reader to Section $5$ of \cite{BondyMurty} for the definition of \emph{a separating vertex, a non-separable graph, a block, and a block tree} of a multi-graph\footnote{Since the work of \cite{Bilo} only considers the class of simple graphs without loops or parallel edges, their definition of a block is slightly different from that for multi-graphs in \cite{BondyMurty}.}.
A \emph{trident} is the graph analogue of a generalized gap $\geq 2$ cutset of cardinality 1:  

\begin{definition}
A  \textit{trident} is a substructure of a connected graph $G$ consisting of connected subgraphs $C, P_1, P_2, P_3$ of $G$ such that (i) each $P_i$ contains exactly one \emph{contact vertex} $s_i$ in common with $C$, (ii) each $P_i$ contains at least two vertices, (iii) [the ``double door" condition] for each $i \neq j$ any path connecting a vertex in $P_i$ with a vertex in $P_j$ passes through both $s_i$ and $s_j$, and (iv):
\begin{enumerate}
	\item[(I)] [type I trident] either $C$ consists of a single vertex $s$ (whence $s=s_1=s_2=s_3$), or
	
	\item[(II)] [type II trident] the three contact vertices are distinct.   
\end{enumerate}
\end{definition}

\begin{definition}  A \emph{bipolar numbering} of a graph $G$ is an enumeration $v_1 , v_2 , \dots , v_k$ of $G$'s vertices satisfying any of the three following (equivalent) conditions:
\begin{itemize}
\item the subgraph induced by any initial or final segment of the list is connected in $G$  
\item for each $j$ with $1< j \leq k$, $v_j$ is adjacent to some $v_i$ with $1\leq i<j$ and for each $m$ with $1 \leq m< k$, $v_m$ is adjacent to some $v_n$ with $m <n\leq k$
\item  for each $j$ with $1< j \leq k$, $G$ contains a path $v_{i_1}, v_{i_2}, \dots , v_{i_r}$ from $v_1$ to $v_j$\\ satisfying  $1 = i_1 < i_2 < \dots < i_r = j$ and for each $m$ with $1 \leq m< k$, $G$ contains a path $v_{i_1}, v_{i_2}, \dots , v_{i_s}$ from $v_k$ to $v_j$ satisfying  $k = i_1 > i_2 > \dots > i_s = j$.\footnote{
The standard definition doesn't mention this third condition.  It can be shown to follow from the second condition via an easy inductive argument.
}
\end{itemize}
\end{definition}

The corresponding notion for tangles is as follows:  
\begin{definition} A bipolar ordering for a tangle $\mathcal{T}$ is a linear order relation $<$ on the elements of $\mathcal{T}$ such that for each $a \in \mathcal{T}$ each of the following sets is (path) connected as a subset of $\mathcal{T}$ :\begin{itemize}
\item $\{ x \in \mathcal{T}\, |\, x < a\}$
\item $\{ x \in \mathcal{T}\, |\, x \leq a\}$
\item $\{ x \in \mathcal{T}\, |\, x > a\}$
\item $\{ x \in \mathcal{T}\, |\, x \geq a\}.$
\end{itemize}

\end{definition}

  The proof of Theorem \ref{thmEQ} 
  consists of a cycle of implications:
\begin{equation}\label{cycle1}
a \Rightarrow b \Rightarrow c \Rightarrow d \Rightarrow^\star e \Rightarrow^\star a.
\end{equation}

\noindent The next result reuses conditions $a$, $d$, and $e$, as well as the two $\star$'d implications, and inserts some key conditions on tangles:

\begin{theorem}\label{Tangle2PreCar}
Let $\mathcal{T}$ be a tangle. 
Let $G^\dagger_{\mathcal{T}}$ be the graph in the topological class 
$\mathcal{G}(\mathcal{T})$ obtained by starting with $G_{\mathcal{T}}$ 
(see Definition \ref{TopClassDef}) and then inserting one additional degree 2 subdivision 
vertex\footnote{
``At least one" would also work, but would be less convenient for the argument.
} along each edge.  Then the following are equivalent:
\begin{enumerate}[label=(\arabic*)]
\item $G^\dagger_{\mathcal{T}}$ admits a bipolar numbering.

\item $\mathcal{T}$ admits a bipolar ordering.

\item $\mathcal{T}$ guarantees connected EF allocations for two agents with monotone continuous valuations. 

\item $\mathcal{T}$ has no gap $\geq 2$ generalized cutset of cardinality 1.

\item The block graph $B(G_{\mathcal{T}})$ is a path. 

\item The block graph $B(G^\dagger_{\mathcal{T}})$ is a path (equivalently, $G^\dagger_{\mathcal{T}}$ does not admit a trident). 


\end{enumerate}

\end{theorem}

\noindent The proof consists of a modified cycle:

\begin{equation} \label{cycle2}
(1) \Rightarrow (2) \Rightarrow (3)  \Rightarrow (4)  \Rightarrow (5) \Rightarrow (6) \Rightarrow^\star (1),
\end{equation}

\noindent in which the $*$'d implication $(6) \Rightarrow^\star (1)$ is lifted directly from 
$d \Rightarrow^\star e \Rightarrow^\star a$ in the proof of Theorem \ref{thmEQ}. 
The desired characterization for tangles and two agents is an immediate corollary:

\begin{corollary}\label{Tangle2Car} [Tangle Characterization for two agents]
 For any tangle $\mathcal{T}$ the following are equivalent:
\begin{enumerate}[label=(\roman*)]

\item $\mathcal{T}$ admits a bipolar ordering.

\item $\mathcal{T}$ guarantees connected EF allocations for two agents with monotone continuous valuations.

\item $\mathcal{T}$ has no gap $\geq 2$ generalized cutset of cardinality 1.

\end{enumerate}

\end{corollary}

\begin{proof} [of Theorem \ref{Tangle2PreCar}]

\noindent [$(1) \Rightarrow (2)$] Our main goal is to convert a bipolar numbering of graph $G^\dagger_{\mathcal{T}}$ into a bipolar ordering of tangle $\mathcal{T}$. Because $\mathcal{T}$ includes actual points along the tangle edges, these points need to be inserted into $G^\dagger_{\mathcal{T}}$'s vertex ordering. But the bipolar numbering induces an orientation of each edge $e=\{ a, b \}$ of $G^\dagger_{\mathcal{T}}$: if $a$ appears before $b$ in the numbering then we orient $e$ from $a$ (source) to $b$ (sink). 
If we ignore, momentarily, the added degree 2 subdivision vertices of $G^\dagger_{\mathcal{T}}$ (so that graph edges correspond to tangle edges), we can now obtain a bipolar ordering of $\mathcal{T}$ by inserting the points along any tangle edge after the source $a$ for that edge and before the sink $b$, with the ordering of points within the edge determined by the ordering of real numbers in the corresponding copy of $[0,1]$, having previously identified $0$ with the edge's source and $1$ with its sink.

More precisely (and with the degree 2 subdivision vertices now accounted for), let $v_1 , v_2 , \dots , v_k$ be a bipolar numbering for $G^\dagger_{\mathcal{T}}$.  For each $r$ with $1 < r \leq k$ let $e_{1,r}, e_{2,r}, \dots , e_{\#(r),r}$ enumerate -- in any order -- all edges $e_{i,r}$ of $G^\dagger_{\mathcal{T}}$ that link $v_r$ to some vertex $v_{s(i,r)}$ appearing earlier in the bipolar numbering (so that $s(i,r) < r$).  

If $v_{s(i,r)}$ corresponds to one of the singular points $v_{s(i,r)}^\star$ of $\mathcal{T}$ then $v_r$ must be one of the added subdivision vertices, and edge $e_{i,r}$ in the graph corresponds to a half tangle edge $e_{i,r}^\star$; without loss of generality we may assume it is the $[0,\frac{1}{2}]$ half of a copy of $[0,1]$, with $v_{s(i,r)}^\star$ identified with the $0$ of this copy. Then, we will use $e_{i,r}^\star$ to designate the interior $(0,\frac{1}{2})$ of this half copy and $v_r ^\star$ will designate the point $\frac{1}{2}$ of this copy.

 If, instead, $v_{s(i,r)}$ corresponds to one of the added subdivision points of $G^\dagger_{\mathcal{T}}$ then $v_r$ must correspond to one of the singular points $v_{r}^\star$ of $\mathcal{T}$. Edge $e_{i,r}$ in the graph again corresponds to a half tangle edge $e_{i,r}^\star$, and without loss of generality we may assume it is the $[\frac{1}{2},1]$ half of a copy of $[0,1]$, with $v_{s(i,r)}^\star$ set (in the previous paragraph) equal to the $\frac{1}{2}$ of this copy. In this case, $e_{i,r}^\star$ will designate the interior $(\frac{1}{2}, 1)$ of this half copy.

Our bipolar ordering for $\mathcal{T}$  will have the following template:

\vspace{-5.5 mm}

\begin{equation}
{\bf v_1^\star} < e_{1,2}^\star < e_{2,2}^\star < \dots  < e_{\#(2),2}^\star < {\bf v_2^\star }< e_{1,3}^\star<  e_{2,3}^\star < \dots  < e_{\#(3),3}^\star < {\bf v_3^\star } \dots < e_{\#(k),k}^\star < {\bf v_k^\star }.
\end{equation}

\noindent Our bipolar ordering $<_{\mathcal{T}}$ extends this template by ordering points within any $e_{i,j}^\star$ according to the ordering of real numbers within the interval $(0, \frac{1}{2})$ or $( \frac{1}{2}, 1)$, depending on which of these intervals has previously been identified with $e_{i,j}^\star$, and respecting the orientation of $e_{i,j}^\star$ as specified earlier.

To see that initial segments of form  $Y = \{ x \in \mathcal{T}\, |\, x <_{\mathcal{T}} a\}$ or $Y=\{ x \in \mathcal{T} \, |\, x \leq _{\mathcal{T}} a\}$ are path-connected in $\mathcal{T}$, it suffices to show that each point $b\in Y$ is path-connected to the first element $v_1^\star$ in the bipolar ordering.

\noindent \underline{Case 1} If $b= v_j ^\star$ for some $j$, use the bipolar numbering of $G^\dagger_{\mathcal{T}}$ to choose a path $v_{i_1}, v_{i_2}, \dots , v_{i_r}$ from $v_1$ to $v_j$ in $G^\dagger_{\mathcal{T}}$ satisfying  $1 = i_1 < i_2 < \dots < i_r = j$. Then, the points $v_{i_1}^\star, v_{i_2}^\star, \dots , v_{i_r}^\star$ all lie in $Y$, as do each of the half tangle edges of form $e_{p,i_t}^\star$ that link $v_{i_t}^\star$ to $v_{s(p,i_t)}^\star$ 
for $1 \leq t < r$. The desired path in $\mathcal{T}$ from $v_1^\star$ to $b$ contains, as its elements, the points   $v_{i_1}^\star, v_{i_2}^\star, \dots , v_{i_r}^\star$ along with all elements of these particular half tangle edges.

\noindent \underline{Case 2} If $b\neq v_j ^\star$ for any $j$ then for some $j < m$, $b$ belongs to the half tangle edge $e_{u,v}^\star$ from $v_j ^\star$  to $v_m ^\star$. Let $P$ be the path in $\mathcal{T}$  from $v_1^\star$ to $v_j^\star$, as constructed in case 1.  To get the desired path for this case, add to $P$ the points $x$ in $e_{u,v}^\star$ satisfying $x \leq b$.

It is immediate from the construction that each such initial segment is a union of finitely many open and closed sets. The argument for  final segments (of form  $Y = \{ x \in \mathcal{T}\, |\, x >_{\mathcal{T}} a\}$ or $Y=\{ x \in \mathcal{T}\, |\, x \geq _{\mathcal{T}} a\}$) is similar, and we skip the details.

\vspace {3mm}

\noindent $[(2) \Rightarrow (3)]$ We apply ``cut and choose." Given a bipolar ordering $<_{\mathcal{T}}$ on $\mathcal{T}$ and a point $x \in \mathcal{T}$ let $L_x$ denote the initial segment $ \{y \in \mathcal{T} \, | \,  y < x\}$. Consider $S = \{x \in \mathcal{T} \, | \, \nu_1(L_x) < \nu_1(\mathcal{T}\setminus L_x) $, where $\nu_1$ denotes the monotone continuous valuation of agent 1.  It then follows, from monotone continuity, that $\nu_1 (S) = \mu_1(\mathcal{T} \setminus S )$, so that for agent 1, the initial segment $S$ and its complement are equally valuable to agent 1.\footnote{
If this construction sounds somewhat overcomplex, keep in mind that there may not exist a point $x \in \mathcal{T}$ with $ \nu_1(L_x) = \nu_1(\mathcal{T} \setminus L_x)$, because $<_{\mathcal{T}}$ is not a complete ordering; the \emph{least upper bound} of $S$ need not exist.
}  For ``cut and choose," agent 1 specifies an initial segment $S$ [worth, to her, the same as its complement], agent 2 chooses either $S$ or $ \mathcal{T} \setminus S $ [whichever has the higher value for her], and 1 gets the remaining piece.

\vspace {3mm}

\noindent $[(3) \Rightarrow (4)]$ This follows immediately as a special case of \ref{GG2L} [the Generalized Gap $\geq 2$ Lemma].

\noindent $[(4) \Rightarrow (5)]$ 
We will show the contrapositive. Assume the block graph $B(G_{\mathcal{T}})$ is not a path. Then, $B(G_{\mathcal{T}})$ contains a vertex $s$ with degree at least three. We have the following two cases: either $s$ is a separating vertex belonging to three distinct blocks of $G_{\mathcal{T}}$, or $s$ is a block containing three distinct separating vertices of $G_{\mathcal{T}}$.

First, suppose that the vertex $s$ is a single separating vertex $v$ of $G_{\mathcal{T}}$ belonging to three distinct blocks $B_i$ ($i = 1,2,3$). It is easy to see that $v$ corresponds to a singular point $v^\star$ of $\mathcal{T}$. In this case, we argue that $\{ \{ v^\star \}\}$ is a gap $\geq 2$ generalized cutset of cardinality 1 for $\mathcal{T}$. Condition $($a$)$ of Definition \ref{GenGapCutsetDef} is immediately satisfied, and we need only to show that deleting $v^\star$ from $\mathcal{T}$ leaves at least three connected components. We know that each of the blocks $B_i$ ($i = 1,2,3$) contains at least one distinct edge $e_i$, and shares $v$ as the unique common vertex in $G_{\mathcal{T}}$. Thus, the subtangles corresponding to each $B_i$ ($i = 1,2,3$) are disconnected from each other after $v^\star$ is deleted from $\mathcal{T}$. 

Second, suppose that the vertex $s$ corresponds to a block $B$ of $G_{\mathcal{T}}$. Then, $B$ is a connected subgraph of $G_{\mathcal{T}}$ and contains at least three distinct separating vertices $s_1,s_2,s_3$, each of which belongs to a distinct block $B_i$ ($i = 1,2,3$). In this case, $B$ corresponds to a subtangle $X$ of $\mathcal{T}$ and $s_1,s_2,s_3$ correspond to the singular points $s^\star_1,s^\star_2,s^\star_3$ of $\mathcal{T}$. Deleting $X$ from $\mathcal{T}$ leaves at least three disconnected components, and the definition of a block tree implies that $\{ \{ X \}\}$ satisfies all requirements for being a gap $\geq 2$ cutset of cardinality 1 for $\mathcal{T}$.

\vspace{3 mm}
\noindent $[(5) \Rightarrow (6)]$ It suffices to prove the following claim, stating that the operation of subdividing a graph preserves the underlying block graph structure. 

\begin{lemma}\label{lem:subdivision}
For a multi-graph $G$ and arbitrary edge $e$, if the block graph $B(G)$ of $G$ is a path, then the block graph $B(G')$ of the graph $G'$ that results from subdividing $e$ (i.e., insert the degree $2$ vertex into $e$) is also a path. 
\end{lemma}
\begin{proof}[of Lemma \ref{lem:subdivision}]
Let $G$ be a multi-graph and $e$ be an edge of $G$. Suppose that $B(G)$ is a path. Let $G'$ be the graph that results from subdividing $e$. If $e$ itself forms a block in $G$, then it is immediate that $B(G')$ continues to be a path. Thus, suppose that $e$ belongs to a block $B$ containing at least three vertices in $G$. Let $B'$ be the subgraph of $G'$ that consists of the original edges in $B$ together with the new subdivided edges $e_1,e_2$, i.e., $B'=(B \setminus \{e\}) \cup \{e_1,e_2\}$. It suffices to show that $B'$ is a block of $G'$. It is known that the subdivision of an edge preserving non-separability, namely, the graph that results from subdividing an edge of a non-separable graph remains non-separable (see Exercise $5.2.1$ in \cite{BondyMurty}). Thus, $B'$ is non-separable. 

To show maximality of $B'$, recall that a graph is non-separable if and only if any two of its edges lie on a common cycle
\cite{Whitney32b}
(see also Theorem $5.2$ of \cite{BondyMurty}). Thus, it suffices to show that if two edges do not belong to a common cycle in $G$, they do not lie on a common cycle in $G'$. Suppose towards a contradiction that two edges $g,h \neq e$ do not lie on a common cycle in $G$ but they belong to a common cycle $C'$ in $G'$. Then, $C'$ needs to include one of the new subdivided edges $e_1$ and $e_2$. Further, the new subdivided vertex incident to both $e_1$ and $e_2$ has degree $2$ in $G'$ and thus $C'$ includes both of the new subdivided edges $e_1$ and $e_2$. Thus, $C'$ corresponds to a cycle $C$ in $G$ that contains $g,h$, a contradiction. A similar argument shows that if the edge $e$ and edge $g \neq e$ do not lie on a common cycle in $G$, then for each $e_i$ $(i=1,2)$, $e_i$ and $g$ do not lie on a common cycle in $G'$. 
Thus, $B'$ remains non-separable and maximal with respect to this property in $G'$.
\end{proof}
Applying the above lemma repeatedly on the edges in $G_{\mathcal{T}}$, we get that $B(G^{\dagger}_{\mathcal{T}})$ is a path if $B(G_{\mathcal{T}})$ is a path.
\end{proof}

\section{Negative transfer from Tangles to Graphs}\label{TanglesToGraphs}
The study of tangles suggests a particular form of structure to impose on the meaning of ``class" as this word is used in the statement of problem $[\clubsuit ]$ (in Section \ref{INTRO})  -- structure that seems appropriate for the setting and allows some progress on the problem.  
Any tangle 
$\mathcal{T}$ 
can be 
thought of as a drawing of a multigraph 
$G_{\mathcal{T}}$, with the singular points and tangle edges of $\mathcal{T}$ becoming the vertices and edges, respectively, of 
$G_{\mathcal{T}}$. For example, for the lips tangle $\mathcal{L}$ the multigraph $G_{\mathcal{L}}$ has three vertices (at $a$, $b$, and $c$ in Figure $1$
) and five edges.  

\begin{definition}\label{TopClassDef} If to any graph $G$  we insert zero or more new vertices of degree 2 along edges, the result is a \emph{subdivision} of $G$, and the added verices are \emph{subdivision vertices}.  
The \emph{topological class associated to} a tangle $\mathcal{T}$ is the infinite collection $\mathcal{G}(\mathcal{T})$ of all subdivisions of $G_{\mathcal{T}}$.
\end{definition}

\noindent Note that for the $[0,1]$ tangle the associated topological class $\mathcal{G}([0,1])$ is the collection of path graphs.  In particular, Theorem \ref{BiloTheorem} can now be rephrased using this new terminology: ``All graphs in $\mathcal{G}([0,1])$ guarantee EF$1_{\emph{outer}}$ allocations for three agents and also for four agents" would be part of that restatement.

We turn next to developing the notion of ``eventually all members" of a topological class -- a notion intended to make precise the idea of ``almost all members" for such a class.  
While our positive results for topological classes of graphs hold for \emph{all} 
graphs in the class (and in this sense resemble Theorem \ref{BiloTheorem}), our negative results do require such a notion. Consider, for example, the following negative result for the lips tangle (which appears stated later as Remark \ref{Lips4}): ``For all positive integers $k \geq 1$ and $n \geq 4$, \underline{eventually all} graphs in the 
topological class of the lips tangle fail to guarantee EF$k_{\emph{outer}}$ 
allocations for $n$ agents."  Replacing the underlined ``eventually all" with ``all" yields an obviously false statement, for the simple reason that any graph whatsoever trivially guarantees  EF$1_{\emph{outer}}$ allocations when the number of agents meets or exceeds the number of vertices.  In the case of the lips tangle, a second reason is that any graph $G \in \mathcal{G}(\mathcal{L})$ whose added subdivision vertices all lie along the two straight and horizontal edges in the lips diagram will be a Hamiltonian graph, thus subject to Theorem \ref{BiloTheorem}.

%

\begin{definition}
[The Subdivision Partial Ordering on Multigraphs] 
\label{SubdivOrderDef} For multigraphs 
 (or \\ graphs) $H$ and $G$,  $H \succeq G$ holds if $H$ is isomorphic to some subdivision of $G$.
\end{definition} 

\begin{proposition} \label{directed}
When restricted to any topological class $\mathcal{G}(\mathcal{T})$ the subdivision order is directed\footnote{That the order is directed when restricted to a topological class relies on including the phrase ``is isomorphic to" in the previous definition \ref{SubdivOrderDef}; alternately, we could have defined a topological class initially to be a set of isomorphism classes.  Ordinarily one sweeps these sorts of subtleties under the rug, but here there are some interesting consequences. For example, one can show that there can be more than one \emph{minimal} $K$ with $K \succeq H_1$ and $K \succeq H_1$; in particular, the directed set may fail to be a lattice.
} -- that is, for any $H_1, H_2 \in \mathcal{G}(\mathcal{T})$, there exists a $K \in \mathcal{G}(\mathcal{T})$ with $K \succeq H_1$ and $K \succeq H_2$.
\end{proposition}

\noindent The easy proof is left to the reader.  The notions introduced in the next definition are only useful when the underlying partial order is directed:

 \begin{definition}
[Final Segments] Let $H \in \mathcal{G}(\mathcal{T})$ be any member of a topological class.  Then $\widehat{\mathbb{H}}$ denotes the set of all $K \in \mathcal{G}(\mathcal{T})$ such $K \succeq H$; such a set $\widehat{\mathbb{H}}$ is called a final segment of $\mathcal{G}(\mathcal{T})$.  A subset $\mathcal{F} \subseteq \mathcal{G}(\mathcal{T})$ of a topological class is final in $\mathcal{G}(\mathcal{T})$ if\, $\widehat{\mathbb{H}} \subseteq \mathcal{F}$ holds for some final segment $\widehat{\mathbb{H}}$ of  $\mathcal{G}(\mathcal{T})$.
\end{definition} 

\begin{definition} 
[Filter] \label{FilterDef} Let $\mathcal{G}$ be a set.  A collection $\mathfrak{F}$ of subsets of $\mathcal{G}$ is a \emph{filter over} $\mathcal{G}$ if
\begin{itemize}
\item $\emptyset \notin \mathfrak{F}$ and $\mathcal{G} \in \mathfrak{F}$, 
\item $Y \in \mathfrak{F}$ whenever $X \subseteq Y \subseteq \mathcal{G}$ and $X \in \mathfrak{F}$, and
\item $X \cap Y \in \mathfrak{F}$ whenever $X \in \mathfrak{F}$ and $Y \in \mathfrak{F}$.
\end{itemize}

\end{definition} 
 
\begin{proposition} \label{FinSegFilter}
Let $\mathcal{G}(\mathcal{T})$ be a topological class of multigraphs and $\mathfrak{F}_{\mathcal{G}(\mathcal{T})}$ be the collection of subsets of $\mathcal{G}(\mathcal{T})$ that are final in $\mathcal{G}(\mathcal{T})$.  Then $\mathfrak{F}_{\mathcal{G}(\mathcal{T})}$ is a filter over $\mathcal{G}(\mathcal{T})$.
\end{proposition}
 
\noindent Proposition \ref{FinSegFilter} follows immediately from the well-known fact that for any set $G$ carrying a directed partial ordering, the collection of subsets of $G$ that are final in that ordering forms a filter. 
 

\begin{definition}
[Eventually All] \label{EventuallyAllDef} Suppose that $\mathbf{P}$ is some property of multigraphs, and that the set of multigraphs in the topological class $\mathcal{G}(\mathcal{T})$ that have property $\mathbf{P}$ belongs to the final segment filter $\mathfrak{F}_{\mathcal{G}(\mathcal{T})}$. In this case, we will say that \emph{eventually all of the members of} $\mathcal{G}(\mathcal{T})$ \emph{are} $\mathbf{P}$.
\end{definition}

\noindent Thanks to the fact that  $\mathfrak{F}_{\mathcal{G}(\mathcal{T})}$ is a filter, it is reasonable to interpret ``eventually all of the members of $\mathcal{G}(\mathcal{T})$ are $\mathbf{P}$" as an assertion that ``almost all multigraphs in $\mathcal{G}(\mathcal{T})$ satisfies property $\mathbf{P}$."  In fact, the definition of filter originated as an attempt to capture precisely a notion of ``almost all."\footnote{
For example, the measure-theoretic notion of ``almost all members of the set $S$ belong to the subset $T \subseteq S$" is that the set difference $S \setminus T$ has measure zero; as long as $\mu (S)>0$ the collection of all such subsets $T$ forms a filter on $S$.   
} To understand why, it helps to run through the three bullets of Definition \ref{FilterDef}, replacing each instance of ``$X \in \mathfrak{F}$" with ``almost all members of $\mathcal{G}$ are members of $X$."   

It now follows that  if eventually all members of $\mathcal{G}(\mathcal{T})$ are $\mathbf{P}$ and eventually all members 
of $\mathcal{G}(\mathcal{T})$ are $\mathbf{Q}$, then eventually all members of $\mathcal{G}(\mathcal{T})$ are (simultaneously) $
\mathbf{P}$ \emph{and} $\mathbf{Q}$.  One consequence is that 
when eventually all members of $\mathcal{G}(\mathcal{T})$ are $\mathbf{P}$  it cannot also be the case that 
eventually all members of $\mathcal{G}(\mathcal{T})$ are
$\neg \mathbf{P}$.  More generally, 
 the collection of all properties that hold for eventually all members of a topological class form a  \emph{logically consistent} and \emph{deductively closed} \emph{``final segment theory"} (which will not be \emph{complete}).  
 
\begin{proposition}
\label{CofNonHam}
Let $\mathcal{T}$ be a tangle.  Then:
\begin{itemize}
\item eventually all  members of \emph{any} topological class $\mathcal{G}(\mathcal{T})$ are graphs (not just multigraphs), 
\item if $\mathcal{T}$ is stringable, then all multigraphs in $\mathcal{G}(\mathcal{T})$ are Hamiltonian, and 
\item  if $\mathcal{T}$ is non-stringable, then eventually all multigraphs in $\mathcal{G}(\mathcal{T})$ are non-Hamiltonian.
\end{itemize}
\end{proposition} 

\noindent We omit the straightforward proof.  Notice that the preceding remarks on final segment theories, applied to this proposition, imply (for example) 
that eventually all members of any non-stringable topological class are both true graphs and non-Hamiltonian; in particular, eventually all members of the topological class $\mathcal{G}(\mathcal{L})$ for lips are non-Hamiltonian graphs.\footnote{This does not imply that $\mathcal{G}(\mathcal{L})$ contains more non-Hamiltonian graphs than Hamiltonian graphs in the sense of \emph{cardinality}, for example.  In fact, as the remarks immediately preceding Definition \ref{SubdivOrderDef} suggest, the cardinalities of these collections are the same -- each is countably infinite.}

With our notion of \emph{eventually all} in hand, we are finally ready to present our main tool for converting results about tangles to results about graphs.  In particular, the Negative Transfer Principle converts a negative result for a tangle into a negative result for eventually all graphs in the corresponding topological class:  

\begin{theorem}\label{NegTransThm} $\mathbf{[Negative}$ $\mathbf{Transfer}$ $\mathbf{Principle]}$ Suppose there is some finite number of agents with monotone continuous valuations for which tangle $\mathcal{T}$ fails to guarantee envy-free, connected allocations.  Then for each positive integer $k$, and each $n$ greater than $\mathcal{T}$'s gap threshold $t$, \emph{eventually all} graphs in the corresponding topological class $\mathcal{G}(\mathcal{T})$ fail to guarantee contiguous EF$k_\emph{outer}$ allocations for $n$ agents.\end{theorem} 

\begin{proof} 
First note that as there is some finite number of agents for which tangle $\mathcal{T}$ fails to guarantee envy-free, connected allocations, $\mathcal{T}$ must be non-stringable, and thus has a finite gap threshold $t$. Let $n$ be any integer with $n \geq t+1$.  Our proof of Lemma \ref{G2L} provides equal valuations $\nu$ over $\mathcal{T}$ for each of $n$ agents, for which there is a strictly positive lower bound $b = \frac{1}{n-1}$ on envy; this means that for every allocation $\{ T_i \} _{1 \leq i \leq n}$ of $\mathcal{T}$ into connected shares there exist agents $i$ and $j$ such that $\nu(T_j) \geq b + \nu(T_i)$ , so that agent $i$ envies $j$'s share ``by at least $b$."

Fix any integer $k \geq 1$. Let $\{ e_r \}_{1 \leq r \leq {\epsilon (\mathcal{T})}}$ enumerate $\mathcal{T}$'s edges.  For each edge, let $J_r$ be an integer that is both greater than $k$ and sufficiently large to make $\frac{\mu (e_r)}{J_r} < \frac{b}{k}$.
Construct a graph $H$ in $\mathcal{G}(\mathcal{T})$ that has $J_r$ subdivision vertices on the edge of $G_{\mathcal{T}}$ corresponding to $e_r$,  for each $r$. Consider the following common valuations for agents over the vertices of $H$: assign value $\frac{\mu(e_r)}{J_r}$ to each subdivision vertex along $[0,1]_r$ and value $0$ to each vertex of $H$ whose degree is not $2$. Note that this discrete distribution mimics the continuous distribution $\mu$ that each agent had over $\mathcal{T}$, replacing the common value $\mu (e_r)$ that each agent placed on $e_r$ with a large number of vertices of common value, placed along this same edge, whose values sum to the original value $\mu (e_r)$.  This number $J_r$ of vertices is high, so that the individual value of any vertex is low -- so low that the total value of any $k$ vertices to any agent is less than $b$.  

We now claim there is no contiguous EF$k_{outer}$ allocation of $H$ for these additively separable valuations of vertices. If such an allocation $\{ A_i \}_{1 \leq i \leq n}$ existed, we could use it to obtain a connected allocation $\{ T_i \}_{1 \leq i \leq n}$  of $\mathcal{T}$ whose value $\mu_i(T_j)$ to each agent $i$ is the same as the value they placed on $A_i$. But then the greatest envy any agent has for another's share is at most the equivalent of the total value of $k$ vertices of $H$, which is less than $b$, a contradiction. The same argument works for any graph $G$ in the final segment $\widehat{\mathbb{H}}$ (for example by giving value $0$ to any additional subdivision vertices in $G$). \end{proof}

\begin{remark}
This argument yields a bound on the number of agents for the class $\mathcal{G}(\mathcal{T})$ of graphs that may be higher than the bound on the number agents for the underlying tangle $\mathcal{T}$.  This is because the  stage of the argument that converts $\{A_i\}_{1 \leq i \leq n}$ into $\{ T_i \}_{1 \leq i \leq n}$ depends on the agents having common valuations over graph vertices,\footnote{Otherwise, different agents would need to break the edges of $\mathcal{T}$ at different locations from each other, in order to obtain shares that agree in value with their shares of $H$.} which rests, in turn, on the common valuation $\nu$ over $\mathcal{T}$ provided in the proof of Lemma \ref{G2L}.  Suppose we knew that the agent bound for tangle $T$ was no higher when we consider only equal valuations for the agents.  In that case, the above argument can be amended\,\footnote{
Instead of deriving the positive lower bound $b$ directly from the Lemma \ref{G2L} proof, this amended version applies a compactness argument to show existence of such a strictly positive bound. 
} to obtain the same agent bound for $\mathcal{G}(\mathcal{T})$ as that assumed for $\mathcal{T}$.
\end{remark}

\section{Discrete lips guarantee EF1 allocations for three agents}\label{StromDiscrete}

We turn our attention to the positive transfer from tangles to graphs: a discrete version of the previous moving knife algorithm returns a contiguous EF$1_\emph{outer}$ allocation for any graph $G$ in the topological class $\mathcal{G}(\mathcal{L})$ associated with lips tangle. Our main theorem of this section is stated as follows.

\begin{theorem}\label{thm:EF1:three}
Every graph in the topological class $\mathcal{G}(\mathcal{L})$ of the lips tangle guarantees EF1$_\emph{outer}$ allocations for three agents who have monotone valuations.  
\end{theorem}

In parallel with the proof in Section \ref{StromLips}, we will use the discrete version of the Stromquist moving-knife algorithm developed in \cite{Bilo}, and show that repeated applications of the algorithm produce EF$1_\emph{outer}$ for three agents. 
While both continuous and discrete procedures are similar in spirit, the discrete version differs in fundamental ways from the version of the continuous algorithm for Lips in Section 4. Indeed, we may run into a trouble if we simply follow the same strategy as in the continuous case: For example, if at the end of the first stage, the median lies at $b$ and the two swords lie at $a$ and $b$, moving one of the swords above may disconnect some piece between the swords and the median, since the swords can only move in discrete steps. Informally, instead of moving the sword over the top of the lips in the latter stages, we will ensure that the sword only hovers over the path at every stage. Applying this repeatedly, we can guarantee that the resulting allocation is EF$1_\emph{outer}$.

We start by explaining that the algorithm of \cite{Bilo} implies the existence of an EF$1_\emph{outer}$ allocation of the specific form. At a high level, in the discrete moving-knife algorithm of \cite{Bilo}, the referee's sword moves from left to right over both vertices and edges, while the agents' knives identify the median lumpy tie for the part lying to right of the sword. In comparison with the original (continuous) version, the additional complexity is that the sword has to stop at one vertex and wait until the knives reach their new lumpy ties by moving step-wise from the previous median lumpy ties. This algorithm produces a contiguous allocation that is envy-free up to the two boundary items as we will see in Theorem \ref{thm:movingknife:Bilo}. 

To formally present the result of \cite{Bilo}, we need the notion of a \emph{median lumpy tie} introduced in \cite{Bilo}, which serves as the discrete analogue of the middle knife. For an enumeration $P=(v_1,v_2,\ldots,v_{r},\ldots,v_m)$ of vertices, $P(v_s,v_t)$ denotes $(v_{s},v_{s+1},\ldots,v_{t})$ when $s \leq t$ and denotes $\emptyset$ when $s > t$.
We define $L(v_r)=P(v_1,v_{r-1})$ to be the sub-enumeration of vertices strictly left of $v_r$ and $R(v_r)=P(v_{r+1},v_{m})$ to be the sub-enumeration of vertices strictly right of $v_r$. For an agent $i$, we say that $v_r$ ($1 \le r \le m$) is a \emph{lumpy tie} over $P$ for $i$ if $r$ is an index such that
\begin{equation*}
	\label{eq:def-lumpy-tie}
	\nu_i(L(v_r) \cup \{ v_r \}) \ge \nu_i(R(v_r)) \quad \text{and}\quad \nu_i(R(v_r) \cup \{ v_r \}) \ge \nu_i(L(v_r)).
\end{equation*}
Suppose that there are thee agents, $i,j,k$. A vertex $v_r$ is a \emph{median lumpy tie} over the sequence $P$ of vertices if there exist vertices $v_a,v_b,v_c$ that are lumpy ties for $i,j,k$, respectively, such that index $r$ is the median of $\{a,b,c\}$. Note that the above definition of a lumpy tie is slightly different from that in \cite{Bilo} in that lumpy ties for an agent may not be uniquely determined. When agents' valuations are monotonic, such vertices should appear consecutively. 

In the appendix, we will present a proof of the following theorem, showing that a slight modification of the algorithm in \cite{Bilo} based on this notion of lumpy tie returns an EF1 allocation of the specific form. 
Here, we consider a different version of EF$1$: For an allocation $A =\{A_i\}_{1 \leq i \leq n}$ and set $X$ of items, we say that agent $i$ \emph{does not envy $j$ up to $X$} if the envy of $i$ towards $j$ can be eliminated by hiding items in $X$, i.e., $\nu_i(A_i) \geq \nu_i(A_j \setminus X)$. Allocation $A$ is \emph{envy-free up to $X$} if 
for any pair of agents $i,j=1,2,\ldots,n$, $i$ does not envy $j$ up to $X$. The discrete moving knife algorithm of \cite{Bilo} returns such an allocation $A$ with $|X \cap A_i| \leq 1$ for each agent $i=1,2,\ldots,n$. 

\begin{theorem}[Implicit in \cite{Bilo}]\label{thm:movingknife:Bilo}
The discrete moving knife algorithm $\mathcal{A}_{discrete}$ for three agents with monotone valuations is applied to the disjoint union of a finite set $I$ of vertices -- called the \emph{initial endowment of the left bundle} -- and an enumeration $P(v_{1},v_m) = \{v_1,v_2,\ldots,v_m\}$ of additional vertices, along with a pair $\ell$, $r$ of indices with initial values $\ell=0$, ${r=r_0}$ such that $v_{r_0}$ is a median lumpy tie over $P(v_{1},v_m)$ and no agent strictly prefers $I$ to both $P(v_{1},v_{r-1})$ and $P(v_{r+1},v_m)$. 
This algorithm increments $r$ at certain stages and $\ell$ at certain other stages, 
and returns an EF$1$ allocation $A$ of $I\cup \{v_1,v_2,\ldots,v_m\}$ that has the following properties:  
\begin{itemize}
    \item[$(${\rm i}$)$] $A$ is EF up to $\{v_{\ell}, v_r\}$ and partitions $I \cup P$ into either $(I \cup P(v_{1},v_{\ell}), P(v_{\ell+1},v_{r}), P(v_{r+1},v_{m}))$
    or $(I \cup P(v_{1},v_{\ell}), P(v_{\ell+1},v_{r-1}),P(v_{r},v_{m}))$; or
    \item[$(${\rm ii}$)$] $A$ is EF up to $\{v_{\ell+1}, v_r\}$, and partitions $I \cup P$ into $(I \cup P(v_{1},v_{\ell}), P(v_{\ell+1},v_{r-1}), P(v_{r},v_{m}))$.
\end{itemize}
The algorithm terminates before or at the time when $\ell$ becomes $m$. If the algorithm does not terminate just after increasing $\ell$ by one, no agent weakly prefers $I \cup P(v_{1},v_{\ell})$ to both $P(v_{\ell+1},v_{r-1})$ and $P(v_{r+1},v_m)$.
Every time just after the algorithm increases $\ell$ by one, $v_r$ is the median lumpy tie over  $P(v_{\ell+1},v_m)$.
Further, if it returns an allocation just after increasing $\ell$ by one, the resulting allocation is of the form $(${\rm i}$)$. Otherwise, the resulting allocation is of the form $(${\rm ii}$)$.\\
\end{theorem}

\noindent Notice that in case of $($i$)$, the items to be hidden belong to the left bundle $I \cup P(v_1,v_{\ell})$ and to either the middle bundle $P(v_{\ell+1},v_{r})$, or the right bundle $P(v_{r},v_{m})$; in case of $($ii$)$, such items belong to the middle bundle $P(v_{\ell+1},v_{r-1})$ and right bundle $P(v_{r},v_{m})$.
 
It is immediate to see that by the above theorem, the resulting allocation of $\mathcal{A}_{discrete}$ is contiguous and EF$1_{\emph outer}$ if $I=\emptyset$ and the given enumeration $P$ is a path.
The following lemma states that the algorithm results in EF$1_\emph{outer}$  as long as the enumeration of vertices is `almost' a path and the sword stops `soon enough.' 
%
In what follows, we denote by $V(G)$ the vertex set of a graph $G$. 
Formally, we say that a graph $G$ admits a \emph{handle decomposition} $(X,G')$ if $G'$ is a subgraph of $G$, $X$ is a path in $G$ where no vertex of $X$ lie in $G'$, and $V(G)=V(G') \cup V(X)$; we call $X$ a \emph{handle} of $G'$. 
For two enumerations $X$ and $Y$ of vertices, $X \cdot Y$ will denote the concatenation of $X$ and $Y$.

\begin{lemma}\label{lem:WOOTOO}
Suppose that a subgraph $G'$ of graph $G$ admits a handle decomposition $(X,G'')$ for which $G''$ admits a bipolar numbering $Y$ over $G''$, where the end vertex of $X$ is adjacent to the first vertex of $Y$. Let $P=X \cdot Y=(v_1,v_2,\ldots,v_m)$. Suppose that $\mathcal{A}_{discrete}$ applies to a set of items $I=V(G) \setminus V(G')$ as an initial endowment,  ordering $P$, and $v_r$ where 
\begin{enumerate}[label=\textup{(\roman*)}]
    \item if $I \neq \emptyset$, some vertex of $I$ is adjacent to $v_1$; and
    \item $v_r$ is a median lumpy tie over $P$; and
    \item no agent strictly prefers $I$ to both $P(v_1,v_{r-1})$ and $P(v_{r+1},v_m)$.
\end{enumerate}
If the algorithm stops before or at the time when $v_{\ell}$ becomes the last vertex of $X$, then the resulting allocation is EF$1_\emph{outer}$. Otherwise, no agent weakly prefers $I \cup V(X)$ to both $P(v_{\ell^*+1},v_{r^*-1})$ and $P(v_{r^*+1},v_{m})$ where ${\ell^*}$ and $r^*$ are the values of indicies ${\ell}$ and $r$, respectively, at the time when $v_{\ell}$ becomes the last vertex of $X$,  and $v_{r^*}$ is a median lumpy tie over $P(v_{\ell^*+1},v_{m})$.
\end{lemma}
\begin{proof}
It is not difficult to see that by definition, the enumeration $P$ that results from concatenating $X$ and $Y$ is a bipolar numbering over the original subgraph $G'$.
We let $X=(u_1,u_2,\ldots,u_k)$. 
If the algorithm $\mathcal{A}_{discrete}$ stops before or during the time when $v_{\ell}=u_{k-1}$, the resulting allocation is contiguous and EF up to $\{v_{\ell}, v_r\}$, or $\{v_{\ell+1}, v_r\}$, by Theorem \ref{thm:movingknife:Bilo}. The resulting bundles of the form:
\begin{itemize}
    \item $(I \cup P(v_{1},v_{\ell}), P(v_{\ell+1},v_{r}), P(v_{r+1},v_{m}))$; or
    \item $(I \cup P(v_{1},v_{\ell}), P(v_{\ell+1},v_{r-1}),P(v_{r},v_{m}))$
\end{itemize}
remain contiguous even after removing $\{v_{\ell}, v_r\}$, or $\{v_{\ell+1}, v_r\}$ since $X$ is a path and $Y$ is a bipolar numbering over $G''$. The same argument applies to the case when the algorithm $\mathcal{A}_{discrete}$ stops at the time when $v_{\ell}$ becomes $u_{k}$; the difference is that the resulting allocation is EF up to $\{v_{\ell}, v_r\}$ so that the middle bundle of the form $P(v_{\ell+1},v_{r-1})$ or $P(v_{\ell+1},v_{r})$ is guaranteed to be contiguous even after removing these items. The second statement immediately follows from Theorem \ref{thm:movingknife:Bilo}. 
\end{proof}

\begin{figure}[ht]
		\centering
				\begin{tikzpicture}[scale=0.8, transform shape,every node/.style={minimum size=6mm, inner sep=1pt}]
				\draw[->, ultra thick] (-6,0) -- (-2,0);
				\draw[gray] (0,0) ellipse (2cm and 1cm);
				\node at (-4,0.6){\large $X$};
                \node[gray] at (0,0.6){\large $Y$};
 
                \begin{scope}[xshift=-130,yshift=5]
		            \draw [rounded corners=0.2mm,fill=black!50] (0.6,-0.8)--(1.1,-0.1)--(0.3,-0.55)--cycle;
		            \draw [fill=black] (0.3,-0.55) -- (0.4,-0.65) -- (0.1,-0.85)-- (0,-0.75);
		        \end{scope}
				\end{tikzpicture}
\caption{Illustration of Lemma \ref{lem:WOOTOO}.}
\label{fig:WOOTOO}
\end{figure}
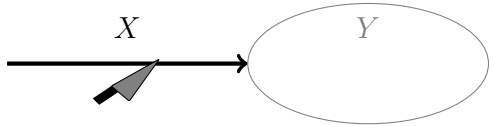

Now, we apply Lemma \ref{lem:WOOTOO} in three stages (each initiated if the algorithm fails to halt in the previous) to prove that contiguous EF$1_\emph{outer}$ allocations always exist for discrete versions of lips.

\begin{proof}[Proof of Theorem \ref{thm:EF1:three}]
Let $G$ be a graph that belongs to the topological class $\mathcal{G}(\mathcal{L})$ of the lips tangle. Define vertices $a,b,c,a_1,b_1,b_2,c_1$ as in Figure \ref{fig:Lips}.
We create the following sequence of subgraphs $G_{i+1}$ together with their handles $X_{i}$ for $i=1,2,3$, where each $X_i$ is adjacent to some vertex of the former $X_j$ $(j<i)$. See Figure \ref{fig:recursion}. We let $G_1=G$. 
\begin{itemize}
    \item Let $X_1$ be a path that starts from $b_1$, goes through the top left path from $b_1$ to $a$, then through the bottom path from $a$ to the vertex just before $c$. Let $G_2$ be the subgraph of $G$ induced by the other vertices. $Y_1$ is a bipolar numbering over $G_2$ that starts from $c$, goes through the top right path from $c$ to $b_2$, then through the middle path from $c_1$ to the right 
    neighbor $a_1$ of $a$. 
    \item Let $X_2$ be the middle left path from $a_1$ to the vertex just before $b$. Let $G_3$ be the top right cycle of the graph $G$. $Y_2$ corresponds to the reverse sub-enumeration of $Y_1$ restricted to the vertices in $G_3$, meaning that $Y_2$ first goes through the middle right path from $b$ to $c_1$ and then through the top right path from $b_2$ to $c$.  
    \item Let $X_3$ be a path over the vertices in the top right cycle that starts from $c$ and ends with the middle left neighbor $c_1$ of $c$. Similarly, $X'_3$ is a path over the vertices in the top right cycle that stars from $b$ and ends with the top right neighbor $b_2$ of $b$. Let $Y_3$ be the empty path and $G_4$ be the empty graph. 

\end{itemize}
The subgraph $G_i$ of $G$ admits a handle decomposition $(X_{i},G_{i+1})$ for $i=1,2,3$. Further, the first vertex of $Y_i$ is adjacent to the last vertex of $X_{i}$ for $i=1,2$.

\begin{figure}[ht]
		\centering
				\begin{tikzpicture}[scale=0.7, transform shape,every node/.style={minimum size=6mm, inner sep=1pt}]
				
				\draw (0,0) to [bend left] (2,2);
				\draw (2,2) to [bend left] (4,0);
				\draw (4,0) to [bend left] (6,2);
				\draw (6,2) to [bend left] (8,0);
				\draw (0,0) to [bend right] (4,-2);
			    \draw (4,-2) to [bend right] (8,0);	
			    
				\node[draw, circle,fill=white](1) at (0,0) {$a$};
				\node[draw, circle,fill=white](2) at (1,0) {$a_1$};
				\node[draw, circle,fill=white](3) at (2,0) {};
				\node[draw, circle,fill=white](4) at (3,0) {};
				\node[draw, circle,fill=white](5) at (4,0) {$b$};
				\node[draw, circle,fill=white](6) at (5,0) {};
				\node[draw, circle,fill=white](7) at (6,0) {};
				\node[draw, circle,fill=white](8) at (7,0) {$c_1$};
				\node[draw, circle,fill=white](9) at (8,0) {$c$};
				
				\node[draw, circle,fill=white](10) at (0.6,1.3) {};
				\node[draw, circle,fill=white](11) at (2,2) {};
				\node[draw, circle,fill=white](12) at (3.4,1.3) {$b_1$};
				\node[draw, circle,fill=white](10) at (4.6,1.3) {$b_2$};
				\node[draw, circle,fill=white](11) at (6,2) {};
				\node[draw, circle,fill=white](12) at (7.4,1.3) {};	
				
				\node[draw, circle,fill=white](13) at (1.3,-1.3) {};
				\node[draw, circle,fill=white](14) at (2.6,-1.9) {};
				\node[draw, circle,fill=white](15) at (4,-2) {};
				\node[draw, circle,fill=white](16) at (5.6,-1.9) {};
				\node[draw, circle,fill=white](17) at (6.9,-1.3) {};
				
				\draw[-, >=latex,thick] (1)--(2) (2)--(3) (3)--(4) (4)--(5) (5)--(6) (6)--(7) (7)--(8) (8)--(9);

				\end{tikzpicture}
\caption{
Lips graphs $G \in \mathcal{G}{(\mathcal{L})}$\label{fig:Lips}}
\end{figure}
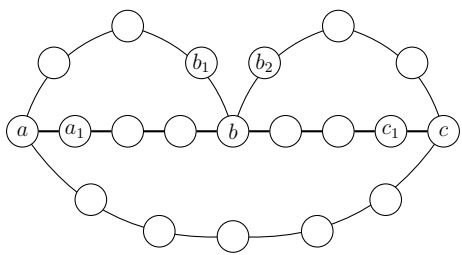
\smallskip

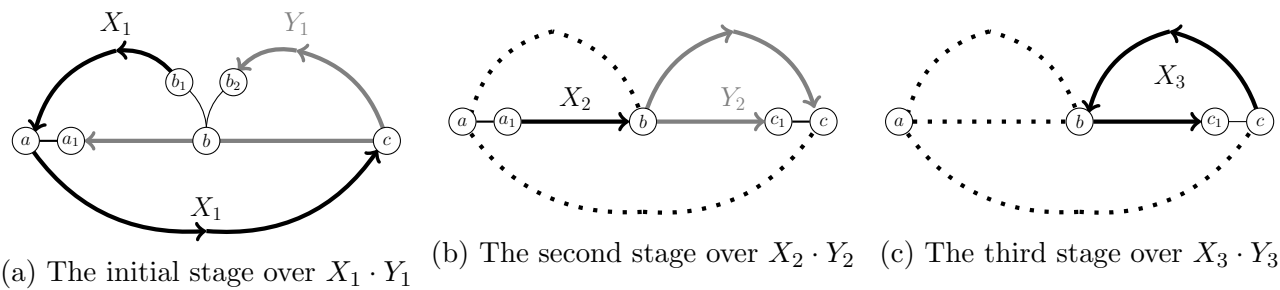
\begin{figure}[ht]
		\centering
\begin{subfigure}{0.33\textwidth}
		\centering
		\begin{tikzpicture}[scale=0.6, transform shape,every node/.style={minimum size=6mm, inner sep=1pt}]
				
				\draw (4,0) to [bend right] (3.4,1.3);
				\draw (4,0) to [bend left] (4.6,1.3);
				
				\draw[->,ultra thick] (3.4,1.3) to [bend right] (2,2);
				\draw[->,ultra thick] (2,2) to [bend right] (0.2,0.2);
				\draw[->,ultra thick] (0,0) to [bend right] (4,-2);
			    \draw[->,ultra thick] (4,-2) to [bend right] (7.8,-0.2);	
			    
			    \draw[->,ultra thick,gray] (8,0) to [bend right] (6,2);
			    \draw[->,ultra thick,gray] (6,2) to [bend right] (4.7,1.55);
			    \draw[->,ultra thick,gray] (8,0) -- (1.3,0) ; 
			    
				\node[draw, circle,fill=white](1) at (0,0) {$a$};
				\node[draw, circle,fill=white](2) at (1,0) {$a_1$};
				\node[draw, circle,fill=white](5) at (4,0) {$b$};
				\node[draw, circle,fill=white](9) at (8,0) {$c$};
				\node[draw, circle,fill=white](12) at (3.4,1.3) {$b_1$};
				\node[draw, circle,fill=white](10) at (4.6,1.3) {$b_2$};

				\draw[-, >=latex,thick] (1)--(2);
				
				\node at (2,2.6){\Large $X_1$};
				\node at (4,-1.5){\Large $X_1$};
                \node[gray] at (6,2.6){\Large $Y_1$};
        \end{tikzpicture}
				\subcaption{The initial stage over $X_1 \cdot Y_1$}%
				\label{fig:firstphase}
\end{subfigure}%
\begin{subfigure}{0.33\textwidth}
		\centering
		\begin{tikzpicture}[scale=0.6, transform shape,every node/.style={minimum size=6mm, inner sep=1pt}]
				
				\draw[loosely dotted,ultra thick] (4,0) to [bend right] (2,2);
				\draw[loosely dotted,ultra thick] (2,2) to [bend right] (0.2,0.2);
                \draw[->,gray,ultra thick] (4,0) to [bend left] (6,2);
				\draw[->,gray,ultra thick] (6,2) to [bend left] (7.8,0.3);
				\draw[loosely dotted,ultra thick] (0,0) to [bend right] (4,-2);
			    \draw[loosely dotted,ultra thick] (4,-2) to [bend right] (7.8,-0.2);	
			    
				\node[draw, circle,fill=white](1) at (0,0) {$a$};
				\node[draw, circle,fill=white](2) at (1,0) {$a_1$};
				\node[draw, circle,fill=white](5) at (4,0) {$b$};
				\node[draw, circle,fill=white](8) at (7,0) {$c_1$};
				\node[draw, circle,fill=white](9) at (8,0) {$c$};
				
				\draw[->,ultra thick] (2)--(5);
				\draw[-, >=latex,thick] (1)--(2) (8)--(9);
				\draw[->,gray,ultra thick] (5)--(6.7,0);
				
				\node at (2.5,0.5){\Large $X_2$};
				\node[gray] at (6,0.5){\Large $Y_2$};
        \end{tikzpicture}
		\subcaption{The second stage over $X_2 \cdot Y_2$}
		\label{fig:secondphase}
		\end{subfigure}%
\begin{subfigure}{0.33\textwidth}
		\centering
		\begin{tikzpicture}[scale=0.6, transform shape,every node/.style={minimum size=6mm, inner sep=1pt}]
				
				\draw[loosely dotted,ultra thick] (4,0) to [bend right] (2,2);
				\draw[loosely dotted,ultra thick] (2,2) to [bend right] (0.2,0.2);
				\draw[->,ultra thick] (8,0) to [bend right] (6,2);
				\draw[->,ultra thick] (6,2) to [bend right] (4.2,0.2);
				\draw[loosely dotted,ultra thick] (0,0) to [bend right] (4,-2);
			    \draw[loosely dotted,ultra thick] (4,-2) to [bend right] (7.8,-0.2);	
			    
				\node[draw, circle,fill=white](1) at (0,0) {$a$};
				\node[draw, circle,fill=white](5) at (4,0) {$b$};
				\node[draw, circle,fill=white](8) at (7,0) {$c_1$};
				\node[draw, circle,fill=white](9) at (8,0) {$c$};

				\node at (6,1){\Large $X_3$};
				\draw[->,ultra thick] (5)--(8);
				\draw[-,loosely dotted,ultra thick] (1)--(5);
				\draw[-] (8)--(9);
        \end{tikzpicture}
		\subcaption{The third stage over $X_3 \cdot Y_3$}
		\label{fig:secondphase}
		\end{subfigure}%
\caption{Discrete moving-knife algorithm. The black arrows correspond to the handles $X_i$ and the gray arrows correspond the bipolar orderings $Y_i$ over $G_{i+1}$. The dotted regions correspond to the set of vertices in $I$ at the beginning of each stage.}
\label{fig:recursion}
\end{figure}

The following procedure iteratively applies the discrete moving-knife algorithm $\mathcal{A}_{discrete}$ while keeping the bundle which is left of the sword contiguous. At each stage, $I$ denotes the initial endowment and $P$ denotes the sequence of vertices to which $\mathcal{A}_{discrete}$ is applied.

\noindent [INITIAL STAGE] Initialize the initial endowment $I=\emptyset$, the sequence $P=X_1 \cdot Y_1$, and the indices $\ell = 0$, and $r$ such that $v_r$ is any median lumpy tie over $P$. Apply the discrete moving-knife algorithm $\mathcal{A}_{discrete}$ to $I$, $P$, $\ell$, and $r$. If the algorithm terminates before or at the time when $v_\ell$ becomes the last vertex of $X_1$, return the resulting allocation. If not, go to the second stage.

\noindent [SECOND STAGE] Initialize $I=V(X_1)$, $P=X_2 \cdot Y_2$, $\ell = 0$, and $v_r$ to the same vertex as the one in the last step of the initial stage. Apply the discrete moving-knife algorithm $\mathcal{A}_{discrete}$ to $I$, $P$, $\ell$, and $r$. If the algorithm terminates before or at the time when $v_\ell$ becomes the last vertex of $X_2$, return the resulting allocation. If not, go to the third stage.

\noindent [THIRD STAGE]
Initialize $I=V(X_1) \cup V(X_2)$. Initialize $r$ so that $v_r$ remains the same as the one in the last step of the second stage. 
\begin{itemize}
\item[(3--1)] If $v_r$ appears strictly after $c_1$ in the numbering $Y_2$ (equivalently, $v_r$ is a top right vertex), then apply the discrete moving-knife algorithm $\mathcal{A}_{discrete}$ to the initial endowment $I$, the sequence $P=X_3$ from $c$ to $c_1$, $\ell = 0$, and $r$. 
\item[(3--2)] If $v_r$ appears strictly before or at $c_1$ in the numbering $Y_2$ (equivalently, $v_r$ is a middle right vertex), then apply the discrete moving-knife algorithm $\mathcal{A}_{discrete}$ to the initial endowment $I$, the sequence $P=X'_3$ from $b$ to $b_2$, $\ell = 0$, and $r$.
\end{itemize}
The algorithm for the third stage 
will terminate before or at the time when $v_{\ell}$ becomes the last vertex of each of the sequences $X_3$ and $X'_3$ by Theorem \ref{thm:movingknife:Bilo}. Note that the last case distinction is necessary so as to ensure as shown below that the sequence $P$ of the third stage is a path for which $v_r$ is a median lumpy tie. 
\smallskip

\noindent
Now, we will show that at the beginning of each stage, the three conditions in Lemma \ref{lem:WOOTOO} are satisfied: 
\begin{enumerate}[label=\textup{(\roman*)}]
    \item $I=\emptyset$ or some vertex of $I$ is adjacent to $v_1$; and
    \item $v_r$ is a median lumpy tie over $P$; and
    \item no agent strictly prefers $I$ to both $P(v_1,v_{r-1})$ and $P(v_{r+1},v_m)$, 
\end{enumerate}
which, together with Lemma \ref{lem:WOOTOO}, then implies that the resulting allocation is contiguous EF$1_\emph{outer}$. 

Consider the beginning of the initial stage. The conditions $($i$)$ and $($ii$)$ hold, by construction of $I$, $P$, and $v_r$. The condition $($iii$)$ also holds by monotonicity and by $I=\emptyset$. 

Consider the beginning of the second stage. The conditions $($i$)$ holds, by construction of $I$ and $P$.
The conditions $($ii$)$ and $($iii$)$ also hold, by the fact that that the algorithm failed to terminate during the previous stage, and by Lemma \ref{lem:WOOTOO}. 

Consider the beginning of the third stage.
Again, $($i$)$ holds, by construction of $I$ and $P$. Indeed, in case (3--1), $v_1=c$ and the vertex just below $c$ is adjacent to $c$ and belongs to $I$. In case (3--2), $v_1=b$ and the left neighbor of $b$ belongs to $I$. 
To see $($ii$)$, at the beginning of the third stage, $v_r$ is a median lumpy tie over $Y_2$ by Lemma \ref{lem:WOOTOO}. If $v_r$ appears strictly after $c_1$ in the numbering $Y_2$, then it is easy to see that $v_r$ is a median lumpy tie over $X_3$, since at least two agents weakly prefer the top right path from $c$ to $v_r$, to the other part in $G_3$, and at least two agents weakly prefer the union of the middle path from $b$ to $c_1$ and the top right path from $b$ to $v_r$, to the rest in $G_3$. Similarly, if $v_r$ appears strictly before or at $c_1$ in the numbering $Y_2$, $v_r$ is a median lumpy tie over $X'_3$. The condition $($iii$)$ holds like in the second stage. 
\end{proof}


\section{Conclusions: envy-freeness for tangles and graphs}\label{Conclusions}

We can now provide partial resolutions to problem $[\clubsuit ]$, restated here:

\medskip

\noindent $[\clubsuit]$ \emph{``For the case of three or more agents, it is a challenging open problem to find an infinite class of non-traceable graphs that guarantee EF1."} 

\medskip

\noindent Does such a class exist?  We provide two opposing answers, depending on how one interprets ``three or more agents."  Informally stated (in particular, one must read ``class" as ``topological class" or as ``eventually all of a topological class") these answers are as follows:

\begin{itemize}
\item \emph{No}: for any $k \geq 1$, the only infinite classes that guarantee EF$k$ for an arbitrary number of agents are classes of Hamiltonian graphs.
\item \emph{Yes}: such an infinite class of non-Hamiltonian graphs exists, for the case of $3$ agents.
\end{itemize}

\noindent The precisely phrased versions of these answers are:

\begin{theorem}\label{BoundedGraphThm}
Let $\mathcal{G}(\mathcal{T})$ be any topological class containing at least one non-Hamiltonian multigraph.  Then there exists a positive integer $n_1$ such that for each two integers  $n > n_1$, and $k \geq 1$, eventually all of the graphs in $\mathcal{G}(\mathcal{T})$ fail to guarantee EF$k_\emph{outer}$ for $n$ agents with monotone valuations.
\end{theorem}

\noindent {\bf Theorem 8.1} \emph{Every graph in the topological class $\mathcal{G}(\mathcal{L})$ of the lips tangle guarantees EF1$_\emph{outer}$ allocations for three agents with monotone valuations.}
\begin{proof} [of \ref{BoundedGraphThm}] 
If $\mathcal{G}(\mathcal{T})$ contains at least one non-Hamiltonian multigraph then by Proposition \ref{CofNonHam}, $\mathcal{T}$ is not a stringable tangle, whence Theorem \ref{OnlyStringable} applies, providing an integer $n_0$ such that $\mathcal{T}$ fails to guarantee connected EF allocations for   $n$ agents whenever $n > n_0$.  Theorem \ref{BoundedGraphThm} now follows immediately from the Negative Transfer Principle (Theorem \ref{NegTransThm}).
\end{proof}

\begin{remark}\label{Lips4} Our negative result \ref{BoundedGraphThm} holds even with very simple valuations (agents have common, additive valuations), while the positive result 8.1 holds for a very broad class of valuations.  Note, as well (and as earlier pointed out) that eventually all graphs in the topological class $\mathcal{G}(\mathcal{L})$ of lips are non-Hamiltonian, providing the desired infinite class of non-Hamiltonian graphs.   
\end{remark}

For \emph{positive} transfers from the continuous realm to the discrete, such as \ref{thm:EF1:three}, we know of no wholesale method analogous to Theorem \ref{NegTransThm}.  Instead, the process operates at the retail level; one takes a proof that connected EF allocations are always possible (for $[0,1]$, for example) and elaborates the argument to guarantee contiguous EFk allocations for a corresponding topological class of graphs (the path graphs, for example).    The moving knife argument of Stromquist \cite{Stromquist} guarantees connected envy-free allocations of $[0,1]$ for three agents, for example, and in \cite{Bilo} this argument is adapted to the discrete setting, to guarantee contiguous EF1$_\emph{outer}$ allocations of any path graph for three agents.\footnote{As a second example of this retail process, \cite{Bilo} also takes the Sperner's lemma proof that connected envy-free allocations of $[0,1]$ for $n$  agents are guaranteed (stated earlier in this introduction as Theorem \ref{UrTheorem}) as the inspiration for a proof that contiguous EF2$_\emph{outer}$ allocations of any path graph for $n$ agents are guaranteed.  Neither of these two adaptations is routine; new ideas and additional details are required.}

In Section \ref{StromLips}, we take the same Stromquist moving knife argument originally applied to $[0,1]$ and show that it can be adapted to guarantee connected envy-free allocations of the lips tangle $\mathcal{L}$ for three agents.  The proof of Theorem  \ref{thm:EF1:three} takes the argument from Section \ref{StromLips} and adapts it to the discrete setting by applying all of the ideas used (in \cite{Bilo}) for the three-agents result for path graphs, along with some new ones.  

\clearpage
\appendix
\section{Appendix: Proof of Theorem \ref{thm:movingknife:Bilo}}

The aim of this subsection is to show that a slight modification of the algorithm in \cite{Bilo}, which takes an initial endowment and a sequence of vertices, produces an EF1 allocation of the desired form. The algorithm in \cite{Bilo} maintains three bundles $L$, $M$, and $R$; intuitively, these bundles are the discrete analogues of the left, middle, and right pieces cut by the sword and the median knife, respectively, in Stromquist's original continuous version \cite{Stromquist}.

We start by introducing a few auxiliary definitions. For a given subsequence $P(v_s,v_t)$ of $P=(v_1,v_2,\ldots,v_m)$ with $s \leq t$ and a vertex $v_r$, we say that agent $i$ is a \emph{left agent} (respectively, \emph{right agent}) over $P(v_s,v_t)$ with respect to $v_r$ if every lumpy tie for $i$ appears strictly before $v_r$. We say that agent $i$ is a \emph{middle agent} over $P(v_s,v_t)$ with respect to $v_r$ if $v_r$ is a lumpy tie for $i$.

Suppose that $v_r$ is a median lumpy tie over the subsequence $P(v_s,v_t)$, and let $i$ be an agent. Then by the definitions of lumpy tie and left and right agents, we have that
\begin{equation}
	\label{eq:left-right-lumpy-tie}
	\begin{array}{l}
	 	\nu_i(L(v_r)) \ge \nu_i(R(v_r) \cup \{ v_r \}) \quad\text{if $i$ is a left agent with respect to $v_r$, and } \\
	 	\nu_i(R(v_r)) \ge \nu_i(L(v_r) \cup \{ v_r \}) \quad\text{if $i$ is a right agent with respect to $v_r$.}
	\end{array}
\end{equation}
Further, if if $v_r$ is a median lumpy tie over $P(v_s,v_t)$, at most one agent can be the left agent; similarly, at most one agent can be the right agent, and at least one agent is a middle agent. 

Given a vertex $v_r$ of the subsequence $P(v_s,v_t)$, and a two-agent set $S$, the function $\Lumpy(S,v_r,P(v_s,v_t))$ identifies agents $i$ and $k$ in $S$ so that $i$'s leftmost lumpy tie appears before or at $k$'s leftmost lumpy tie (breaking ties arbitrarily), and returns the allocation of the items in $P(v_s,v_t)$ to $S$ such that 
\begin{itemize}
\item if $i$ is a left agent, then $i$ receives $L(v_r)$ and $k$ receives $R(v_r) \cup \{ v_r \}$; 
\item if $i$ is a middle agent, then agent $k$ receives $k$'s preferred bundle among $L(v_r)$ and $R(v_r)$, and agent $i$ receives the other bundle along with $v_r$. 
\end{itemize}
We show that $\Lumpy(S,v_r,P(v_s,v_t))$ returns an EF1 allocation when $v_r$ is a median lumpy tie. 

\begin{lemma}[Median Lumpy Ties Lemma (Lemma $9$ in \cite{Bilo})]\label{lem:lumpy}
Suppose that there are three agents. Let $S=\{i,k\}$ be a pair of agents, and $v_r$ be a median lumpy tie over the subsequence $P(v_s,v_t)$. Then $\Lumpy(S, v_r,P(v_s,v_t))$ is an EF allocation of $P(v_s,v_t)$ to $S$ up to $\{v_r\}$. Furthermore, each agent in $S$ receives a bundle weakly better than the two bundles $L(v_r)$ and $R(v_r)$.
\end{lemma}
\begin{proof}
Let $A=\{A_i,A_k\}$ be an allocation returned by $\Lumpy(S, v_r,P(v_s,v_t))$. If $i$ is a left agent, neither of these agents envies the other by \eqref{eq:left-right-lumpy-tie} and by the definition of a lumpy tie.
If $i$ is a middle agent, she does not envy agent $k$ by the definition of a lumpy tie and $k$ envies $i$ up to $\{v_r\}$ by construction. 
\end{proof}

\noindent\fbox{%
\parbox{\linewidth}{%
\noindent 
\textbf{Discrete moving-knife algorithm $\mathcal{A}_{discrete}$ for a set $N$ of three agents}
over the disjoint union of a finite set $I$ of vertices and an enumeration $P(v_{1},v_m) = \{v_1,v_2,\ldots,v_m\}$ of additional vertices, along with a pair $\ell$, $r$ of indices with initial values $\ell=0$, ${r=r_0}$ such that $v_{r_0}$ is a median lumpy tie over $P(v_{1},v_m)$ and no agent strictly prefers $L=I$ to both $M=P(v_{\ell +1},v_{r_0-1})$ 
and $R=P(v_{r+1},v_m)$: 
\smallskip
An agent $i$ is a \emph{shouter} if $i$ weakly prefers $L$ to both $M$ and $R$.
\def\stromquistscale{0.81}
\begin{enumerate}[leftmargin=40pt, rightmargin=2pt]
\item[\emph{Step 1.}] Delete the left-most point of the middle bundle, i.e., set $M=\{v_{\ell+2},v_{\ell+3},\ldots,v_{r-1}\}$. \\
If the number of shouters is smaller than two, go to Step 2.
If at least two agents shout, we show that there is a shouter $\shouter$ who is a middle agent over $P(v_{\ell+1},v_m)$ with respect to $v_r$ (see Lemma \ref{lem:middle}). Then, allocate $L$ to a shouter $\shoutleft$ distinct from $\shouter$. Let the agent $\chooser$, who is distinct from $\shouter$ and $\shoutleft$, choose his preferred bundle among $\{v_{\ell+1}\}\cup M$ and $\{v_r\}\cup R$. Agent $\shouter$ receives the other bundle.
\item[\emph{Step 2.}]
If $v_r$ is the median lumpy tie over $P(v_{\ell+2},v_m)$, directly move to the following cases (a)--(c). 
If $v_r$ is not the median lumpy tie over $P(v_{\ell+2},v_m)$, set $r=r+1$, $M=\{v_{\ell+2},v_{\ell+3},\ldots,v_{r-1}\}$, and $R=\{v_{r+1},v_{r+2},\ldots,v_{m}\}$; then, consider the following cases (a)--(d).
\begin{enumerate}
    \item If at least two agents shout, find a shouter $\shouter$ who did not shout at the previous step. If there is a shouter $\shoutleft$ who shouted at the previous step, $\shoutleft$ receives $L$; else, give $L$ to an arbitrary shouter $\shoutleft$ distinct from $\shouter$. The agent $\chooser$, distinct from $\shouter$ and $\shoutleft$, chooses his preferred bundle among $\{v_{\ell+1}\} \cup M$ and $\{v_r\} \cup R$, breaking ties in favor of the former option. Agent $\shouter$ receives the other bundle.
    \item If $v_r$ is a median lumpy tie over $P(v_{\ell+2},v_m)$ but at most one agent shouts, go to Step 3.
    \item Otherwise $v_r$ is not the median lumpy tie over $P(v_{\ell+2},v_m)$: Repeat Step 2.
\end{enumerate}
\item[\emph{Step 3.}] Add an additional item to $L$, i.e., set $\ell=\ell+1$ and $L=I \cup \{v_1,v_2,\ldots,v_\ell\}$. \\ 
If no agent shouts, go to Step 1. 
If there is a shouter $\shoutleft$ who shouted at the previous step, $\shoutleft$ receives $L$; else, give $L$ to an arbitrary shouter $\shoutleft$. 
Allocate the remaining items according to $\Lumpy(N\setminus \{\shoutleft\},v_r,P(v_{\ell+1},v_m))$.
\end{enumerate}
}%
}

\smallskip
\noindent
Before we proceed to proving Theorem \ref{thm:movingknife:Bilo}, we observe the following auxiliary lemmas. 

\begin{lemma}\label{lem:middle}
If $\mathcal{A}_{discrete}$ terminates at Step 1, there is a shouter $\shouter$ who is a middle agent over $P(v_{\ell+1},v_m)$ with respect to $v_r$. Further, 
\begin{align}
\nu_{\shouter}(\{v_{\ell+1}\}\cup M) > \nu_{\shouter}(L) \geq \nu_{\shouter}(M),\label{eq:middle1}\\ \nu_{\shouter}(\{v_r\} \cup R) \geq \nu_{\shouter}(\{v_{\ell+1}\}\cup M)>\nu_{\shouter}(R).\label{eq:middle2}
\end{align}
\end{lemma}
\begin{proof}
Consider any agent $i$ who shouted in Step 1. If this is the first time when Step 1 is implemented, $i$ strictly prefers $\{v_{\ell+1}\} \cup M$ or $R$ to $L$. Similarly, if the last step just before is Step 3, $i$ strictly prefers $\{v_{\ell+1}\} \cup M$ or $R$ to $L$ because $i$ did not shout at the previous step. In either case, we have that
\begin{align*}
\max\{\nu_{i}(\{v_{\ell+1}\}\cup M),\nu_i(R) \} > \nu_{i}(L) \geq \max\{\nu_{i}(M),\nu_i(R) \}, 
\end{align*}
where the second inequality holds because $i$ shouted in Step 1. Thus, $i$ strictly prefers $\{v_{\ell+1}\}\cup M$ to $R$. Thus, the left-most item of $R$ cannot be a lumpy tie for each of the two shouters, and so every lumpy tie over $P(v_{\ell+1},v_m)$ for the shouters appears strictly before or at $v_r$. Since there are two agents who shouted and at most one agent can be the left agent, at least one shouter is a middle agent. The inequalities \eqref{eq:middle1} and \eqref{eq:middle2} hold by the above argument and by the fact that $\shouter$ is a middle agent. 
\end{proof}

\begin{lemma}\label{lem:shouter}
If $\mathcal{A}_{discrete}$ terminates at Step 2(a), and $i$ is a shouter who did not shout in the previous step, then
\begin{align}\label{eq:step4:new-shouter}
\nu_i(\{v_r\}\cup R) > \nu_i(L) \ge \nu_i(M).
\end{align}
\end{lemma}
\begin{proof}
Observe that the last step before Step 2(a) did not change the middle bundle $M$ and increased $r$ by $1$, deleting the left-most item of the previous right bundle $\{v_r\}\cup R$. Thus, since $i$ is not a shouter at the previous step and $i$ is a shouter at Step 2(a), we have that  
\[
\max\{\nu_i(M),\nu_i(\{v_r\}\cup R)\} >  \nu_i(L) \geq  \max \{\nu_i(M),\nu_i(R)\}, 
\]
which implies the desired inequality \eqref{eq:step4:new-shouter}.  
\end{proof}

\noindent
Now, we are ready to prove Theorem \ref{thm:movingknife:Bilo}.

\begin{stmnt*}[Theorem \ref{thm:movingknife:Bilo}]
The discrete moving knife algorithm $\mathcal{A}_{discrete}$ for three agents with monotone valuations is applied to the disjoint union of a finite set $I$ of vertices -- called the \emph{initial endowment of the left bundle} -- and an enumeration $P(v_{1},v_m) = \{v_1,v_2,\ldots,v_m\}$ of additional vertices, along with a pair $\ell$, $r$ of indices with initial values $\ell=0$, ${r=r_0}$ such that $v_{r_0}$ is a median lumpy tie over $P(v_{1},v_m)$ and no agent strictly prefers $I$ to both $P(v_{1},v_{r-1})$ and $P(v_{r+1},v_m)$. 
This algorithm increments $r$ at certain stages and $l$ at certain other stages, and returns an EF$1$ allocation $A$ of $I\cup \{v_1,v_2,\ldots,v_m\}$ that has the following properties:  
\begin{itemize}
    \item[$(${\rm i}$)$] $A$ is EF up to $\{v_{\ell}, v_r\}$ and partitions $I \cup P$ into either $(I \cup P(v_{1},v_{\ell}), P(v_{\ell+1},v_{r}), P(v_{r+1},v_{m}))$
    or $(I \cup P(v_{1},v_{\ell}), P(v_{\ell+1},v_{r-1}),P(v_{r},v_{m}))$; or
    \item[$(${\rm ii}$)$] $A$ is EF up to $\{v_{\ell+1}, v_r\}$, and partitions $I \cup P$ into $(I \cup P(v_{1},v_{\ell}), P(v_{\ell+1},v_{r-1}), P(v_{r},v_{m}))$.
\end{itemize}
The algorithm terminates before or at the time when $\ell$ becomes $m$. If the algorithm does not terminate just after increasing $\ell$ by one, no agent weakly prefers $I \cup P(v_{1},v_{\ell})$ to both $P(v_{\ell+1},v_{r-1})$ and $P(v_{r+1},v_m)$. Every time just after the algorithm increases $\ell$ by one, $v_r$ is the median lumpy tie over $P(v_{\ell+1},v_m)$. \\
Further, if it returns an allocation just after increasing $\ell$ by one, the resulting allocation is of the form $(${\rm i}$)$. Otherwise, the resulting allocation is of the form $(${\rm ii}$)$.
\end{stmnt*}
\begin{proof}[Proof of Theorem \ref{thm:movingknife:Bilo}]
To see that the algorithm is well-defined: by Lemma \ref{lem:middle}, there is a shouter $\shouter$ who is a middle agent over $P(v_{\ell+1},v_m)$ when at least two agents shout in Step 1. The algorithm terminates and returns an allocation before or at the time when $\ell$ becomes $m$: when $\ell$ becomes $m$, both of the middle and right bundles becomes empty and every agent shouts because of monotonicity. By construction, the algorithm also ensures that 
\begin{itemize}
    \item every time it increases $\ell$ by $1$ in Step 3, $v_r$ is a median lumpy tie over $P(v_{\ell+1},v_m)$, and
    \item if it does not terminate when increasing $\ell$ by $1$ at Step 3, no agent weakly prefers $L=P(v_1,v_{\ell})$ to both $R=P(v_{\ell+1},v_{r-1})$ and $M=P(v_{r+1},v_{m})$. 
\end{itemize}
We will show that the resulting allocation $A$ is 
\begin{itemize}
    \item EF up to $\{v_{\ell+1},v_{r}\}$ if the algorithm terminates at Step 1, or Step 2(a). 
    \item EF up to $\{v_{\ell},v_{r}\}$ if the algorithm terminates at Step 3. 
\end{itemize}

\noindent
To see this, we will consider each of the three steps at which the algorithm might terminate. 

\noindent
\textbf{Step 1.} Suppose that the algorithm terminates at Step 1. We claim that each agent does not envy others up to $\{v_{\ell+1}\}$ or $\{v_r\}$.
\begin{itemize}[leftmargin=18pt]
	\item Agent $\shoutleft$ does not envy the agent receiving the bundle $\{v_{\ell+1}\} \cup M$ up to $\{v_{\ell+1}\}$ because $\shoutleft$ shouted and receives $L$. 
	Similarly, $\shoutleft$ does not envy the agent who receives bundle $\{v_r\} \cup R$ up to $\{v_r\}$.
	\item Agent $\chooser$ does not envy agent $\shoutleft$; 
	In both cases when this is the first time when Step 1 is implemented and also when Step 1 follows an instance of Step 3, agent $\chooser$ weakly prefers $\{v_{\ell+1}\} \cup M$ or $R$ to $L$. 
	 Further, $\chooser$ gets his preferred bundle among $\{v_{\ell+1}\} \cup M$ and $\{v_r\} \cup R$. Thus, $\chooser$ does not envy $\shoutleft$. \\	
	Agent $\chooser$ does not envy agent $\shouter$ since $\chooser$ obtains his preferred bundle among $\{v_{\ell+1}\} \cup M$ and $\{v_r\} \cup R$. 
	\item Agent $\shouter$ is a middle agent, and by Lemma \ref{lem:middle}, we have that
    \begin{align*}
        &\nu_{\shouter}(\{v_{\ell+1}\}\cup M) > \nu_{\shouter}(L) \geq \nu_{\shouter}(M),\\
        &\nu_{\shouter}(\{v_r\} \cup R) \geq \nu_{\shouter}(\{v_{\ell+1}\}\cup M)>\nu_{\shouter}(R).
    \end{align*}
    Thus, if $\shouter$ receives bundle $\{v_{\ell+1}\}\cup M$, then she does not envy $\shoutleft$, and does not envy $\chooser$ up to $\{v_r\}$. If $\shouter$ receives bundle $\{v_r\} \cup R$, then she does not envy the other agents. 
\end{itemize}

\noindent
\textbf{Step 2(a).} Suppose that the algorithm terminates at Step 2. We claim that each agent does not envy others up to $\{v_{\ell+1}\}$ or $\{v_r\}$.

\begin{itemize}[leftmargin=18pt]
	\item Agent $\shoutleft$ does not envy the agent who gets the bundle $\{v_{\ell+1}\} \cup M$ up to $\{v_{\ell+1}\}$ because $\shoutleft$ shouted and receives $L$. Similarly, $\shoutleft$ does not envy the agent who receives bundle $\{v_r\} \cup R$ up to $\{v_r\}$.
	
	\item Agent $\chooser$ gets his preferred bundle among $\{v_{\ell+1}\} \cup M$ and $\{v_r\} \cup R$, and hence does not envy agent $\shouter$ who receives the other bundle. Also, we claim that agent $\chooser$ does not envy $\shoutleft$: If $\chooser$ is not a shouter, she strictly prefers either $M$ or $R$ to $L$ and hence does not envy $\shoutleft$. If $\chooser$ is a shouter, by the choice of $\chooser$, she did not shout in the previous step; applying the inequality \eqref{eq:step4:new-shouter} in Lemma \ref{lem:shouter}, we have that 
	\[
	\max\{\nu_{\chooser}(\{v_r\} \cup R),\nu_{\chooser}(\{v_{\ell+1}\} \cup M)\}>\nu_{\chooser}(L),
	\]
	meaning that $\chooser$ does not envy $\shoutleft$. 
	\item Agent $\shouter$ does not envy others up to $\{v_{\ell+1}\}$ or $\{v_r\}$: First, consider the case when agent $\shouter$ receives bundle $\{v_r\} \cup R$. Since $\shouter$ is a shouter who did not shout in the previous step, we can apply \eqref{eq:step4:new-shouter} in Lemma \ref{lem:shouter} and obtain
	\[
	\nu_{\shouter}(\{v_r\} \cup R) >\nu_{\shouter}(L) \geq \nu_{\shouter}(M).
	\]
	Thus, $\shouter$ does not envy $\shoutleft$ who receives $L$ and does not envy $\chooser$ who receives $\{v_{\ell+1}\} \cup M$ up to $\{v_{\ell+1}\}$. Second, consider the case when agent $\shouter$ receives bundle $\{v_{\ell+1}\} \cup M$. This means that $\chooser$ strictly prefers $\{v_r\} \cup R$ over $\{v_{\ell+1}\} \cup M$ by the tie breaking rule. Recall that the last step before Step 2(a) incremented $r$ by $1$ since the vertex $v_{r-1}$ is not a median lumpy tie over $P(v_{\ell+2},v_m)$. Thus, $\chooser$ cannot be a left or middle agent over $P(v_{\ell+1},v_m)$ with respect to $v_{r^*}$ where $v_{r^*}$ corresponding to the previous median lumpy tie $v_r$ over $P(v_{\ell+1},v_m)$ in Step 1 just before the algorithm enters this iteration of Step 2. Hence, $\shouter$ is either a left or middle agent over $P(v_{\ell+1},v_m)$ with respect to $v_{r^*}$; and by definition of a median lumpy tie,  
	\[
	 \nu_{\shouter}(\{v_{\ell+1}\} \cup M) \geq \nu_{\shouter}(\{v_{\ell+1},v_{\ell+2},\ldots, v_{r^*}) \geq \nu_{\shouter}(\{v_{r^*+1},v_{r^*+2},\ldots,v_m\}) \geq \nu_{\shouter}(\{v_r\} \cup R).
	\]
	Combining this with \eqref{eq:step4:new-shouter}, we conclude that $\shouter$ does not envy the other agents. 
\end{itemize}

\noindent
\textbf{Step 3.} Suppose that the algorithm terminates at Step 3. We claim that each agent does not envy others up to $\{v_{\ell}\}$ or $\{v_r\}$.
\begin{itemize}[leftmargin=18pt]
	\item Since $\shoutleft$ is a shouter, agent $\shoutleft$ obtains $L$ and does not envy the other agents up to $\{v_r\}$.
	\item Agent $i$ who is not a shouter does not envy $\shoutleft$ because $i$ strictly prefers either $M$ or $R$ to $L$, and thus by Lemma \ref{lem:lumpy} receives a bundle preferred to $L$.
	\\ \smallskip
	Agent $i$ does not envy the other agent $j \neq \shoutleft$ up to $\{v_r\}$ by Lemma \ref{lem:lumpy}.
	\item We claim that agent $i \neq \shoutleft$ who is a shouter does not envy $\shoutleft$ up to $\{v_\ell\}$. Recall that Step 3 follows an instance of Step 2(b), and by definition of Step 2(b) at most one agent shouted at this step. Thus, by the choice of $\shoutleft$, $i$ did not shout at the previous step and hence $i$ strictly prefers either $M$ or $R$ to the previous left bundle $L\setminus \{v_{\ell}\}$ of Step 2(b). By Lemma \ref{lem:lumpy}, agent $i$ receives a bundle weakly preferred to $M$ or $R$. Thus, $i$ does not envy $\shoutleft$ up to $\{v_\ell\}$.\\
	Also by Lemma \ref{lem:lumpy}, agent $i$ does not envy the other agent $j \neq \shoutleft$ up to $\{v_r\}$.
\end{itemize}
We conclude that the allocation returned by any of the steps satisfies the desired properties.
\end{proof}

\section{Appendix: Blocks and block graphs}
Consider a multi-graph $G$. A \emph{separation} of a connected graph is a decomposition of the graph into two non-empty connected subgraphs $G_1$ and $G_2$ that has only one vertex in common. The common vertex is called a \emph{separating vertex} of the graph. A graph $G$ is called \emph{non-separable} if it is connected and has no separating vertices; otherwise, it is called \emph{separable}. 
A \emph{block} of a graph is a subgraph which is non-separable and is maximal with respect to this property. It is known that the blocks of a connected graph form a tree. 

\begin{lemma}[\cite{BondyMurty}]
Let $G$ be a graph. Then:
\begin{itemize}
\item any two blocks of $G$ have at most one vertex in common, 
\item the blocks of $G$ form a decomposition of $G$, 
\item each cycle of $G$ is contained in a block of $G$. 
\end{itemize}
\end{lemma}

For a multi-graph $G$, consider a bipartite graph $B(G)$ with bipartition $(\mathcal{B},S)$, where $\mathcal{B}$ is the set of blocks of $G$ and $S$ the set of separating vertices of $G$; a block $B$ and a separating vertex $v$ is adjacent in $B(G)$ if and only if $v$ belongs to $B$. The block graph $B(G)$ is does not contain a cycle by the above lemma. The graph $B(G)$ is therefore referred to as the \emph{block tree} of $G$. 

\section{Appendix: Monotone continuous valuations on tangles }

Here we provide the detailed definition of monotone continuity for valuations on tangles.  This assumption is what we need for the proof of Theorem \ref{StromLipsTheorem} and Proposition 2.5. It is weaker than assuming that valuations are countably additive, nonatomic measures -- in particular, nothing resembling additivity or sub-modularity is presumed.

Start with any (continuous and bijective) identification $\mathcal{I}$ of each of the $m$ tangle edges $e_i$ of a tangle $\mathcal{T}$ with a copy $[0, 1]_i$ of the unit interval. Then $\mathcal{I}$ induces a corresponding map from $\Re ^{2m}$ to closed subsets of $\mathcal{T}$, wherein any $z = (x_1,y_1,x_2,y_2, \dots x_{2m},y_{2m}) \in \Re ^{2m}$ is mapped to the closed (but not necessarily connected) subset $X = X_{\mathcal{I}}(z)$ for which $X \cap e_i$ corresponds to $[x_i , y_i] \cap [0,1]_i$ for each $i = 1,2, \dots , m$.


\begin{definition}  A \emph{tangle valuation} $\nu$ for a tangle $\mathcal{T}$ assigns a non-negative real number $\nu (X)$ to every connected subset $X$ of $\mathcal{T}$; $\nu$ is  \emph{monotone continuous} if it is both:
\begin{enumerate}
\item \emph{monotone:} $\nu(\emptyset) = 0$ and $v(X) \leq v(Y)$ whenever $X \subseteq Y$ are connected subsets of $\mathcal{T}$, and

\item \emph{continuous:} The function $z \mapsto v(X_{\mathcal{I}}(z))$, restricted to the subset of $\Re^{2m}$ corresponding to  closed and connected subsets of $\mathcal{T}$, is continuous as a function of $2m$ real variables.
\end{enumerate}
\end{definition} 

\begin{remark} \label{jumpyremark} Continuity of $\mathcal{I}$ and its inverse implies that our definition is independent of the choice of $\mathcal{I}$.  Note, as well, that monotone continuous valuations are necessarily 
 \emph{non-jumpy:} $v(X) = v(X^\circ)$ for every connected subset $X$ of
  $ \mathcal{T}$ (where $X^\circ$ denotes the \emph{interior} of $X$, in 
  $\mathcal{T}$'s topology).  For any connected subset $X$ of a tangle, the
   set difference $X \setminus X^\circ$ contains finitely many points.
     Now suppose $\mu$ is a countably additive measure.  Then if $\mu$ is non-atomic, it follows that it is non-jumpy.  More generally,  continuity of $\mu$ follows from countable additivity + non-atomicity and monotonicity follows from additivity along with the standard non-negativity assumption for measures. 
 \end{remark}

\end{document}